\documentclass[11pt]{amsart}

\usepackage[left=1in,right=1in,top=1in,bottom=1in]{geometry}

\usepackage{microtype}  

\usepackage{mathtools}
\usepackage{bbm}  

\swapnumbers

\usepackage{hyperref}

\usepackage{amssymb,amsfonts,amsmath,amsthm}
\usepackage[mathscr]{eucal}
\usepackage[alphabetic]{amsrefs}
\usepackage{stmaryrd}

\usepackage[ps,matrix,arrow,curve,cmtip]{xy}

\title{Looijenga line bundles in complex analytic elliptic cohomology}
\author{Charles Rezk}
\date{ \today}
\address{Department of Mathematics \\
University of Illinois \\ 
Urbana, IL}
\email{rezk@illinois.edu}
\thanks{The author was supported by NSF grant DMS-1406121.}

\numberwithin{equation}{section}

\makeatletter
  \let\c@subsection\c@equation
\makeatother

\theoremstyle{plain}   

\newtheorem{thm}[subsection]{Theorem}
\newtheorem{prop}[subsection]{Proposition}
\newtheorem{cor}[subsection]{Corollary}
\newtheorem{lemma}[subsection]{Lemma}

\theoremstyle{remark}
\newtheorem{rem}[subsection]{Remark}    
\newtheorem{exam}[subsection]{Example}

\theoremstyle{plain}

\begin{document}


\newcommand{\margnote}[1]{\mbox{}\marginpar{\tiny\hspace{0pt}#1}}

\def\lambada{\lambda}


\newcommand{\id}{\operatorname{id}}
\newcommand{\colim}{\operatorname{colim}}
\newcommand{\llim}{\operatorname{lim}}
\newcommand{\Cok}{\operatorname{Cok}}
\newcommand{\Ker}{\operatorname{Ker}}
\newcommand{\Image}{\operatorname{Im}}
\newcommand{\op}{{\operatorname{op}}}
\newcommand{\Aut}{{\operatorname{Aut}}}
\newcommand{\End}{{\operatorname{End}}}
\newcommand{\Hom}{{\operatorname{Hom}}}

\newcommand*{\ra}{\rightarrow}
\newcommand*{\lra}{\longrightarrow}
\newcommand*{\xra}{\xrightarrow}
\newcommand*{\la}{\leftarrow}
\newcommand*{\lla}{\longleftarrow}
\newcommand*{\xla}{\xleftarrow}

\newcommand{\ho}{\operatorname{ho}}
\newcommand{\hocolim}{\operatorname{hocolim}}
\newcommand{\holim}{\operatorname{holim}}

\newcommand*{\realiz}[1]{\left\lvert#1\right\rvert}
\newcommand*{\len}[1]{\left\lvert#1\right\rvert}
\newcommand{\set}[2]{{\{\,#1\mid#2\,\}}}
\newcommand*{\tensor}[1]{\underset{#1}{\otimes}}
\newcommand*{\pullback}[1]{\underset{#1}{\times}}
\newcommand*{\powser}[1]{[\![#1]\!]}
\newcommand*{\laurser}[1]{(\!(#1)\!)}
\newcommand{\ndiv}{\not|}
\newcommand{\pairing}[2]{\langle#1,#2\rangle}

\newcommand{\F}{\mathbb{F}}
\newcommand{\Z}{\mathbb{Z}}
\newcommand{\N}{\mathbb{N}}
\newcommand{\R}{\mathbb{R}}
\newcommand{\Q}{\mathbb{Q}}
\newcommand{\C}{\mathbb{C}}

\newcommand{\point}{{\operatorname{pt}}}
\newcommand{\Map}{\operatorname{Map}}
\newcommand{\eev}{\wedge}
\newcommand{\sm}{\wedge} 

\newcommand*{\mc}{\mathcal}
\newcommand*{\mf}{\mathfrak}
\newcommand*{\mr}{\mathrm}
\newcommand*{\ul}{\underline}
\newcommand*{\ol}{\overline}
\newcommand*{\wt}{\widetilde}
\newcommand*{\wh}{\widehat}

\newcommand{\dfn}{\textbf}

\def\noloc{\;{:}\,}

\def\defeq{\overset{\mathrm{def}}=}

\newcommand{\forcepar}{\mbox{}\par}

\newcommand{\Sip}{\Sigma^\infty_+}
\newcommand{\Oi}{\Omega^\infty}

\newcommand{\T}{\mathbb{T}}

\newcommand{\gh}{\mathrm{gh}}

\newcommand{\Borel}{\mathrm{Borel}}

\newcommand{\ch}{\operatorname{ch}}

\newcommand{\Ell}{\mathrm{Ell}}
\newcommand{\EllTate}{\mathrm{Ell}_\mathrm{Tate}}
\newcommand{\Spec}{\operatorname{Spec}}
\newcommand{\an}{\mathrm{an}}
\newcommand{\der}{\mathrm{der}}

\newcommand{\Top}{\mathrm{Top}}
\newcommand{\Diff}{\operatorname{Diff}}

\newcommand{\Sp}{\operatorname{Sp}}

\newcommand{\QF}{\mathrm{QF}}

\newcommand{\cp}{\operatorname{Cup}}
\newcommand{\pp}{\operatorname{Pont}}
\newcommand{\act}{\operatorname{Act}}
\newcommand{\mult}{\operatorname{Mult}}
\newcommand{\inv}{\operatorname{Inv}}

\newcommand{\Ab}{\mathrm{Ab}}

\newcommand{\sym}{\mathrm{sym}}
\newcommand{\antisym}{\mathrm{antisym}}

\newcommand{\Sym}{\operatorname{Sym}}

\renewcommand{\Im}{\operatorname{Im}}
\newcommand{\Sh}{\mathrm{Sh}}
\newcommand{\Coh}{\mathrm{Coh}}

\newcommand{\An}{\mathrm{An}}

\newcommand{\ext}{\mathrm{ext}}

\newcommand{\tf}{\mc{P}}

\newcommand{\bbss}{\,\backslash\!\backslash\,}

\newcommand{\Lift}{\operatorname{Lift}}

\newcommand{\Lie}{\operatorname{Lie}}

\begin{abstract}
We present a calculation, which shows how the moduli of complex
analytic elliptic curves arises naturally from the Borel cohomology of
an  extended moduli space of $U(1)$-bundles on a torus.  Furthermore,
we show how the analogous calculation, applied to a moduli 
space of principal bundles for a $K(\Z,2)$ central extension of
$U(1)^d$ give rise to Looijenga line bundles.  We then 
speculate on the relation of these calculations to the construction of
complex analytic 
equivariant elliptic cohomology. 
\end{abstract}

\maketitle


\section{Introduction}

In this note, we describe some aspects of how complex analytic
elliptic curves arise naturally from the cohomology of certain spaces
which parameterize principal bundles on orientable genus 1 surfaces.
This suggests how elliptic cohomology emerges from certain derived
complex analytic spaces associated to dimensional reduction applied to
2-dimensional field theories.

\subsection{Complex analytic elliptic cohomology}

Complex analytic equivariant elliptic cohomology was first defined by
Grojnowski \cite{grojnowski-delocalized-elliptic}\footnote{Originally
  circulated as  a preprint in 1994; see
  \cite{ando-miller-grojnowskis-ell-coh}.}.  In its most basic 
formulation, given 
\begin{itemize}
\item
a compact connected abelian Lie group $G$
(i.e., $G\approx U(1)^d$), with cocharacter lattice $B=\Hom(U(1),G)$,
and 
\item 
an elliptic curve $C_\tau = \C/\Z\tau+\Z$ for $\Im\tau>0$, 
\end{itemize}
he obtains an equivariant coholomolgy theory
\[
\Ell_G^* \colon h\Top_G^{\mr{fin}}  \ra \Coh(C_\tau\otimes B) 
\]
on  $G$-spaces homotopy equivalent to finite $G$-CW-complexes, taking
values in coherent sheaves of  $\mc{O}_{C_\tau\otimes
  B}$-modules on the complex analytic abelian variety $C_\tau\otimes B\approx
C_\tau^d$.  

Grojnowski descibes his construction as ``declocalized''.  That is,
$\Ell_G^*(X)$ 
is produced by gluing together certain localizations of the values of
Borel equivariant 
cohomology rings $H^*(X^H\times_G EG;\C)$ for various subgroups $H$ of
$G$.  Conceptually, one can 
regard this as a ``reverse engineered'' version of a character sheaf,
by analogy with the interpretation of $\C\otimes K_G(X)$ as a sheaf
over the multiplicative group $\mathbb{G}_m$, whose localizations at
various points 
of $\mathbb{G}_m$ are computed, in terms of standard localization
theorems, in terms of Borel cohomology (e.g., as in
\cite{baum-brylinski-macpherson-delocalisee}).   

Grojnowski's theory has been extended and used to  explain aspects of
elliptic genera, notably the rigidity of the Ochanine genus
\cite{rosu-elliptic-rigidity}, and the modularity of the Witten genus
\cite{ando-basterra-witten-genus},
\cite{ando-sigma-orientation-analytic}.   A significant  feature of
this  theory is the ablility to twist by
a \emph{level}, which in the above formulation is described in terms
of tensoring sheaves with 
the \emph{Looijenga line bundle} associated to a quadratic form on the
cocharacter lattice $B$
\cite{grojnowski-delocalized-elliptic}*{\S3.3}.  Looijenga's theta
functions appear 
explicitly in the Kac character formula, which can be identified with
the calculation of a Gysin map in elliptic cohomology 
\cite{ando-power-op-loop-group}, \cite{ganter-elliptic-weyl}.  

This construction of analytic elliptic cohomology, though productive,
is somewhat ad hoc, and technically rather intricate.  Furthermore, we
should expect more from the theory.  In particular,
\begin{enumerate}
\item it should take values not (merely) in sheaves on a scheme or
  complex analytic space, but rather in sheaves on a \emph{derived}
  scheme or 
  complex analytic space, and 
\item it should in some sense classify \emph{two-dimensional
    reductions} of certain 
  kinds   2-dimensional field theories.
\end{enumerate}
These should nowadays be much more approachable 
goals than was the case when Grojnowski originally defined the theory.
For point (1), there is  well-developed machinery for 
constructing cohomology theories from \emph{derived} geometric objects
\cite{lurie-elliptic-survey}.  Furthermore, there is a direct
construction of a derived algebraic scheme realizing rational equivariant
elliptic cohomlogy for $G=U(1)$ following Grojnowski's delocalized
approach \cite{greenlees-rational-equivariant-elliptic}.  
 Point (2) is more difficult; however,
there has been partial success in relating elliptic cohomology to
field theory (following the program of Segal
\cite{segal-elliptic-cohomology}), and many features  of the
relationship are  understood  (see, e.g.,
\cite{stolz-teichner-susy-and-gen-coh}).

\subsection{This purpose and results of this paper}
\label{subsec:results}

We are motivated by the observation that elliptic cohomology at the
Tate curve should be associated to 
\emph{one-dimensional reduction} of 1-dimensional field theories.
Very roughly, Tate elliptic cohomology should arise as some kind of
equivariant $K$-theory for extended loop groups.  By the ``extended
loop group'' $\mc{L}^{\mathrm{ext}}G$ of $G$, we really   
mean the topological groupoid whose objects are certain principal $G$-bundles
$P\ra \T$, over a circle $\T=\R/\Z$, and whose morphisms are maps $(P\ra \T)\ra
(P'\ra \T)$ of $G$-bundles covering a rotation of the circle.  

Our point
of view is inspired by that of \cite{ganter-stringy-power},
\cite{ganter-power-op-orbifold-tate-k}, which considers the special
case of finite groups $G$, in which case $\mc{L}^{\mathrm{ext}}G$ is a Lie
groupoid, and thus comes with a well-defined equivariant $K$-theory.
Furthermore, Kitchloo has defined a version of equivariant $K$-theory
for certain 
Kac-Moody groups \cite{kitchloo-dominant-k-theory}.  Using this, he
constructs \cite{kitchloo-quantization-modular-functor-eq-ell}, for
loop groups on simple and simply connected $G$, a version of
$G$-equivariant elliptic cohomology associated to the Tate curve.
It turns out that Looijenga line bundles arise naturally in this framework.

The purpose of this note is to describe calculations inspired by the
idea of \emph{two-dimensional reduction}.  Thus, (i) the circle $\T$ is
replaced with an orientable genus 1 surface 
$\Sigma$ (e.g., $\T^2$), and (ii) equivariant $K$-theory is replaced with Borel
cohomology with complex coefficients.  We restrict attention to a
limited class of equivariance groups $G$, namely (i) tori $G=U(1)^d$,
or (ii) ``central extensions'' $\wt{G}=U(1)^d\times_\phi K(\Z,2)$ of a
torus $G$ by $K(\Z,2)$,  
according to a class $\phi\in H^4(BG;\Z)$.  

We summarize our calculations as follows; precise statements
are given in \S\S\ref{sec:analytic-moduli}--\ref{sec:looijenga}.  Fix 
\[
\Sigma=\text{orientable genus 1 surface},\qquad G=\text{topological
  group},
\]
and consider the ``wreath product'' group
\[
\mc{W}(G)=\mc{W}^\Sigma(G):= \Map(\Sigma, G)\rtimes \Diff(\Sigma),
\]
where $\Diff(\Sigma)$ is the group of diffeomorphisms\footnote{That we
  use the diffeomorphism group here is not significant, since we will only
  use homotopy invariant features of this action.  Thus, in its place
  we could use the homeomorphism group of $\Sigma$, or even the monoid
  of self-homotopy equivalences of $\Sigma$, each of which have the
  same homotopy type as $\Diff(\Sigma)$.}
(not
necessarily orientation preserving).
Note that $\Map(\Sigma, G)$ is the gauge group of $\Sigma\times G\ra
\Sigma$,  the trivial
$G$-bundle over $\Sigma$, and thus $\mc{W}(G)$ is an ``extended gauge
group''.  Its classifying space $B\mc{W}(G)$ is thus a homotopy theoretic
moduli space for the data (smooth genus 1 surface, principal
$G$-bundle). 

Let $\mc{W}_0(G)\subseteq \mc{W}(G)$ denote the identity component,
with discrete quotient $\ol{\mc{W}}(G)=\mc{W}(G)/\mc{W}_0(G)$.  Thus
there is a natural action $\ol{\mc{W}}(G)\curvearrowright
B\mc{W}_0(G)$ on the classifying space of the connected subgroup.
For the $G$ we will consider, the cohomology ring 
$H^*(B\mc{W}_0(G);\C)$ is concentrated in even degree, whence we
obtain an action
\[
(\ol{\mc{W}}(G)\times \C^\times)^\op \;\curvearrowright \; \Spec H^*(B\mc{W}_0(G);\C)
\]
on an affine complex variety, where $\C^\times$ acts linearly on
$H^2$.  Let 
\[
\mc{X}_G := [\Spec H^*(B\mc{W}_0(G);\C)]_\an \smallsetminus
\{\text{bad}\},
\]
which is a complex analytic space obtained as the ``analytification''
of the complex variety, with a certain closed subset (described in 
\S\ref{subsec:geometric-picture}) removed.  In our
examples $\mc{X}_G$ is always smooth.  The object we are interested in
is
\[
\mc{M}_G := (\ol{\mc{W}}(G)\times \C^\times) \bbss
\mc{X}_G, 
\]
the stacky quotient in complex manifolds.  We compute that
\begin{align*}
  \mc{M}_e &\approx \mc{M} = \text{the moduli stack of (complex analytic)
             elliptic curves,}
\\ 
  \mc{M}_{U(1)} &\approx  \mc{E} = \text{the universal elliptic curve over
                  $\mc{M}$,}
\\
  \mc{M}_{U(1)^d} &\approx \mc{E}^d = \mc{E}\times_{\mc{M}}\cdots
                    \times_{\mc{M}}\mc{E} 
                    = \text{the $d$-fold product of $\mc{E}$,}
\\
  \mc{M}_{K(\Z,2)} &\approx \mathbb{G}_m\times \mc{M} = \text{the
                     multiplicative group as a trivial bundle of
                     groups over
                     $\mc{M}$,}
\\ 
  \mc{M}_{U(1)^d\times_\phi K(\Z,2)} &\approx \mc{P}_\phi = \text{principal
                           $\mathbb{G}_m$-bundle associated to $\mc{L}_\phi$}
\end{align*}
where $\mc{L}_\phi\ra \mc{E}^d$ is the ``Looijenga line bundle''
associated to $\phi\in H^4(BU(1)^d,\Z)$, regarded as a quadratic
function $\phi\colon H_2BU(1)^d=\Z^d\ra \Z$.  The first three cases of
the computation are easy observations, and are described
in \S\ref{sec:analytic-moduli}.  The main purpose of this paper is
prove the last two cases, which  are stated in
\S\ref{sec:looijenga}.

\subsection{Organization of this paper}

The basic observation is the following:
the universal complex analytic elliptic curve arises \emph{naturally}
from the cohomology of such moduli spaces.  We present this
observation in \S\ref{sec:analytic-moduli}.  I have not seen this
observation stated in this way before; however, it is closely related
to an observation by Etingof and Frenkel about coadjoint actions in
double loop groups, a relationship we describe briefly in
\S\ref{subsec:remarks-double-loop}.  

In \S\ref{sec:looijenga}, we replace $G=U(1)^d$ with
$\wt{G}=U(1)^d\times_\phi K(\Z,2)$, the extension associated to a class
$\phi\in H^4(BG;\Z)$, and observe that our formulation naturally gives
Looijenga-type line bundles.  This is stated as  \eqref{thm:main-thm},
which is our main result.

In \S\ref{sec:isogenies} we observe how isogenies of complex analytic
elliptic curves fit naturally into this story, via finite covering
maps of genus 1 surfaces.

In \S\ref{sec:remarks-on-formalism} we speculate as to how these constructions
might give rise to \emph{derived} elliptic curves (in an analytic
setting) following the pattern described in
\cite{lurie-elliptic-survey}, and to elliptic cohomology theories of
Grojnowski type. 
We only sketch a picture here; setting
this up formally would involve confronting a definition of derived
complex analytic space, which is beyond the scope of this note.  We
also describe the ``stacky'' dependence of our constructions on the
group $G$, and note what happens in the simpler 1-dimensional case
(where $\Sigma$ is a circle).

The remainder of the paper
(\S\ref{sec:spaces-in-2-3}--\S\ref{sec:proof-of-theorem}) 
is taken up with the proof of the main result \eqref{thm:main-thm},
which is itself a derived from a more general and coordinate invariant
formulation \eqref{thm:general-thm}.

\subsection{Conventions}

At various points we need to consider the action of a group on another
group (always from the left).  We will
sometimes use the notation $g\propto h$ for such an action, so as to
typographically distinguish it from $gh$ a product of group elements.
When $G$ acts on $H$ from the left, a semidirect product $K$ is always
a group with subgroups  $G$
and $H$  that $GH=HG=K$ and $G\cap H=\{1\}$, and such
that $ghg^{-1}=g\propto h$.   There are two distinct but 
canonically isomorphic constructions of such:
$G\ltimes H$ and $H\rtimes G$ with group laws $(g,h)\cdot(g',h')=(gg',
(g'^{-1}\propto h)h')$ and $(h,g)\cdot (h',g')=(h(g\propto h'), gg')$
respectively.
In \S\ref{sec:conventions}, we describe the homotopy theoretic
conventions we use, primarily in order to establish the sign
conventions we need in
(\S\ref{sec:spaces-in-2-3}--\S\ref{sec:proof-of-theorem}).

\subsection{Acknowledgements}

I would like to thank Matt Ando and Dan Berwick-Evans for stimulating
conversations which have helped direct the shape of this work.

\section{Analytic moduli of elliptic curves, vs.\ homotopic
   moduli of genus $1$ surfaces}
\label{sec:analytic-moduli}

\subsection{Moduli  of elliptic curves over $\C$}

The classical uniformization theory of Weierstrass says that
\begin{enumerate}
\item every elliptic curve is isomorphic, as a complex manifold, to
  $\C/\Lambda$ for some lattice $\Lambda$, (i.e., a subgroup $\Lambda =
  \Z t_1+\Z t_2$ such that $\R\otimes \Lambda=\C$), with neutral
  element at the origin, and 
\item every map $\C/\Lambda\ra \C/\Lambda'$ between such complex
  manifolds fixing the neutral element is given by multiplication by a
  complex scalar.
\end{enumerate}
That is, such curves correspond to lattices in $\C$ up
to scaling by a non-zero complex number.

This can be enriched to a desciption of the moduli stack of such curves.
Let 
\[
\mc{X} := \set{(t_1,t_2)}{\R t_1+\R t_2=\C} \subset \C^2. 
\]
We have a group action 
\[
GL_2(\Z)\times \C^\times\;\curvearrowright\; \mc{X}
\] by 
\begin{equation}\label{eq:sl2z-on-X}
A\propto (t_1,t_2) = (at_1+bt_2, ct_1+dt_2),\qquad A=(\begin{smallmatrix}
  a&b\\ c&d\end{smallmatrix})\in GL_2(\Z),
\end{equation}
and 
\begin{equation}\label{eq:Ctimes-on-X}
\lambda\propto (t_1,t_2) = (\lambda t_1,\lambda t_2), \qquad \lambda\in
\C^\times. 
\end{equation}
Points of the quotient space $(GL_2(\Z)\times \C^\times)\backslash
\mc{X}$ are in bijective correspondence to homothety-equivalence
classes ($\Lambda\sim \lambda\Lambda$) of lattices, i.e., to 
isomorphism classes of
elliptic curves.  It turns out that the moduli stack is in fact the
stack 
quotient 
\[
\mc{M} := (GL_2(\Z)\times \C^\times) \bbss \mc{X}.
\]
For our purposes, we do not need to worry about the general notion of
stacks.  It is sufficient to remember that information defining
$\mc{M}$ is precisely contained in the group action, so that (for
instance), sheaves on the stack $\mc{M}$ are precisely equivariant
sheaves on $\mc{X}$.  

\begin{rem}
The stack $\mc{M}$ is an orbifold, though the above does not present
it as such.  The continuous group $\C^\times$ acts freely on
$\mc{X}$, so that $\C^\times\backslash \mc{X}\approx
\C\smallsetminus \R$ defined by $(t_1,t_2)\mapsto
\tau=t_1/t_2$ gives an identification with the double-half plane.  The
residual $GL_2(\Z)$-action descends to an 
action on $\C\smallsetminus\R$ with finite isotropy, whence
$\mc{M}\approx GL_2(\Z)\bbss
(\C\smallsetminus\R)$.   
\end{rem}

\begin{rem}
Instead of $\mc{X}$ we could use  $\mc{X}^+=\set{(t_1,t_2)\in
  \mc{X}}{\Im(t_1/t_2)>0}$, so 
  $\mc{M}\approx (SL_2(\Z)\times \C^\times)\bbss \mc{X}^+\approx
  SL_2(\Z)\bbss \mc{H}$ where $\mc{H}=\set{\tau\in \C}{\Im\tau>0}$.
\end{rem}

\subsection{The universal elliptic curve}

The universal elliptic curve $\mc{E}\ra \mc{M}$ can be modelled by a
map $C\ra \mc{X}$, with fiber $C_{(t_1,t_2)} = \C/(\Z t_1+\Z t_2)$
over $(t_1,t_2)\in \mc{X}$, together with a lift of the group action
on $\mc{X}$.   Since the fibers are themselves 
quotients by a free action, we can decribe the universal curve as a
stack quotient, via the action
\[
( GL_2(\Z)\ltimes \Z^2)\times \C^\times \;\curvearrowright\; \mc{X}\times
\C = \set{(t_1,t_2,y)\in \C^2\times \C}{\R t_1+\R t_2=\C}
\]
defined by 
\begin{align}\label{eq:action-on-XtimesC}
\begin{aligned}
A\propto (t_1,t_2,y) &= (at_1+bt_2, ct_1+dt_2,y), 
\\
(m_1,m_2) \propto (t_1,t_2,y) &= (t_1,t_2,y+m_1t_1+m_2t_2), 
\\
\lambda\propto (t_1,t_2,y) &= ( \lambda t_1, \lambda
t_2,\lambda y),
\end{aligned}
&&
\begin{aligned} 
A\in GL_2(\Z),
\\
(m_1,m_2)\in \Z^2,
\\
\lambda\in \C^\times.
\end{aligned}
\end{align}
Thus, the stack quotient 
\[
\mc{E} := ((GL_2(\Z)\ltimes \Z^2)\times \C^\times)
\bbss \mc{X}\times \C
\]
presents the universal curve over $\mc{M}$.

\subsection{Moduli of genus $1$ surfaces}

Fix a smooth surface $\Sigma$, closed and orientable of genus $1$.  We write
\[
\Diff(\Sigma)\supset  \Diff_0(\Sigma),
\]
for the group of diffeomorphisms and its identity component.
The classifying space
$B\Diff(\Sigma)$ can be viewed as a homotopy-theoretic moduli space
of orientable (but not oriented) genus $1$-surfaces.

For convenience in describing calculations, we we use the model
$\Sigma:=\T^2=\R^2/\Z^2$. 
Then $\Diff(\Sigma)$ is weakly equivalent, as a topological group,
to the subgroup $\T^2\rtimes GL_2(\Z)$ (acting on $\T^2$ in the
evident way from 
the left) \cite{earle-eells-diffeo-cpt-riemann}.  

Therefore $B\Diff_0(\Sigma)\approx B\T^2$, which carries an evident
action by 
$GL_2(\Z)=\Diff_+(\Sigma)/\Diff_0(\Sigma)$.   We may thus consider the
induced action 
\[
GL_2(\Z)^\op\; \curvearrowright\; H^*(B\Diff_0(\Sigma);\C). 
\]
It is immediate that
\[
H^*(B\Diff_0(\Sigma);\C)= H^*(B\T^2;\C)\approx \C[t_1,t_2],
\qquad t_1,t_2 \in H^2, 
\]
with $GL_2(\Z)$ action given by the precisely the formula
\eqref{eq:sl2z-on-X}.  The cohomology also carries a natural
$\C^\times$ action, determined by the grading, which coincides with
\eqref{eq:Ctimes-on-X}. 

\subsection{Universal degree $0$ line bundle on a genus $1$ surface}

Now consider the group 
\[
\mc{W}(U(1)):= \Map(\Sigma, U(1))\rtimes \Diff(\Sigma),
\]
which has identity component $\mc{W}_0(U(1)):= \Map_0(\Sigma,U(1))
\rtimes \Diff_0(\Sigma)$, and set $\ol{\mc{W}}(U(1)):=
\pi_0\mc{W}(U(1)) = \mc{W}(U(1))/\mc{W}_0(U(1))$.  The classifying
space $B\mc{W}(U(1))$ carries the universal example of a degree 0
complex line bundle over a smooth genus 1 surface.  

Using  $\Sigma=\T^2$, we obtain an explicit finite dimensional model
for $\mc{W}(U(1))$ (up to homotopy equivalence), namely
\[
\bigl(\Hom(\T^2,U(1)) \times U(1)\bigr) \rtimes \bigr(GL_2(\Z)\ltimes 
\T^2\bigr). 
\]
That is, the homomorphism $\Hom(\T^2,U(1))\times U(1)\ra
\Map(\Sigma,U(1))$ defined by\footnote{We write the group laws on
  $\T^2=\R^2/\Z^2$ and $U(1)\approx \R/\Z$ additively, and use the
  evident isomorphism $\Z^2=\Z^{1\times 2}\approx \Hom(\T^2,U(1))$.}  
$(m,y)\mapsto
((s_1,s_2)\mapsto  
y+m_1s_1+m_2s_2)$ is a homotopy equivalence, and is invariant under the
evident action of $GL_2(\Z)\rtimes \T^2\subset \Diff(\Sigma)$.  We can
rebracket this as 
\[
\bigl(GL_2(\Z)\ltimes \Z^2\bigr) \ltimes \bigl(\T^2\times U(1)\bigr),
\]
using the left action $GL_2(\Z)\ltimes \Z^2 \curvearrowright
\T^2\times U(1)$ given by $(A,m)\propto (t,y) = (At, y+m_1t_1+m_2t_2)$.  

The induced action $\ol{\mc{W}}(U(1))^\op \curvearrowright
H^*(B\mc{W}_0(U(1));\C)$ thus has the form
\[
(GL_2(\Z)\ltimes \Z^2)^\op\;\curvearrowright\; H^*(B(\T^2\times U(1));\C).
\]
We easily read off that
\[
H^*(B(\T^2\times U(1));\C) \approx \C[t_1,t_2,y],\qquad t_1,t_2,y\in H^2,
\]
with $GL_2(\Z)\ltimes \Z^2$ action given precisely by the first two
formulas from 
\eqref{eq:action-on-XtimesC}\footnote{We can regard cohomology
  classes ``$t_1$'', ``$t_2$'' and ``$y$'' as coordinate functions on the space
  $\mc{X}\times \C=\{(t_1,t_2,y)\}$, so the formulas of
  \eqref{eq:action-on-XtimesC} also describe how to pull back such
  functions.}.  The grading of cohomology corresponds 
to the $\C^\times$-action from \eqref{eq:action-on-XtimesC}.  

\subsection{The geometric picture}
\label{subsec:geometric-picture}

As in the introduction (\S\ref{subsec:results}) we write
\[
\mc{W}(G)= \mc{W}^\Sigma(G) := \Map(\Sigma,G) \rtimes \Diff(\Sigma),
\]
with group law $(\psi,\phi)\cdot
(\psi',\phi') = (\psi\cdot (\psi'\circ\phi^{-1}), \phi\circ \phi')$,
for the extended gauge group of a trivial principal $G$-bundle over
$\Sigma$; hence the classifying space $B\mc{W}(G)$ carries the
universal example of a 
trivializable principal $G$-bundle over a genus 1 surface.  We let
$\mc{W}_0(G)\subseteq \mc{W}(G)$ denote the identity component, and set
$\ol{\mc{W}}(G) = \mc{W}(G)/\mc{W}_0(G)= \pi_0\mc{W}(G)$.  

Now assume that we restrict to groups $G$ for which
$H^*(B\mc{W}_0(G);\C)$ is concentrated in even degrees.  
We obtain an action
\[
(\ol{\mc{W}}(G) \times \C^\times)^\op \;\curvearrowright \;
H^*(B\mc{W}_0(G);\C)
\]
where $\C^\times$ acts by scalar multiplication on $H^2$, and note
that this action 
is \emph{functorial} with respect to the group $G$ and homomorphisms,
i.e., $\phi\colon G\ra G'$ induces a map of cohomology rings which is
compatible with the the group actions in the evident way.  In
particular, the tautological homomorphism $G\ra e$ induces a map
$\pi\colon B\mc{W}_0(G)\ra B\Diff_0(\Sigma)$ which is
invariant under the action of $\ol{W}(G)$.  

We can now take the analyticification of the resulting affine scheme
over $\C$.  Define
\[
\mc{X}_G :=  [\Spec H^*(B\mc{W}_0(G);\C)]_\an \smallsetminus B_G,
\]
where $B_G$ is the closed (in the analytic topology) subset consisting
of $\C$-points $p$ such that the composite
\[
H^2(B\Diff_0(\Sigma);\R)\ra H^2(B\Diff_0(\Sigma);\C) \xra{\pi^*}
H^2(B\mc{W}_0(G);\C) \xra{p} \C
\]
is \emph{not} a bijection.  Thus $\mc{X}_G$ is the preimage of
$\mc{X}=\mc{X}_e\subset [\Spec H^*(B\Diff_0(\Sigma);\C)]_\an\approx
\C^2$ with respect to the map induced by $\pi$, and is invariant under
the action of $\ol{\mc{W}}(G)\times \C^\times$.  Hence we obtain
\[
\mc{M}_G := \ol{\mc{W}}(G)\times \C^\times \bbss \mc{X}_G.
\]

\subsection{Products of elliptic curves vs.\ degree $0$ torus-bundles}

Consider $G=U(1)^d$, with $d\geq1$.  As in the case of $d=1$, we have
a finite dimensional model
\[
\bigl( \Hom(\T^2,U(1)^d)\times U(1)^d) \bigr) \rtimes \bigl(GL_2(\Z)
\ltimes \T^2)\bigr) \xra{\sim} \mc{W}(U(1)^d)
\]
which can be rebracketed as 
\[
\bigl(GL_2(\Z) \ltimes \Z^{d\times 2}\bigr) \ltimes \bigl(\T^2\times
U(1)^d\bigr).  
\]
Thus
\[
H^*(B\mc{W}_0(U(1)^d);\C) \approx H^*(B(\T^2\times U(1)^d);\C) \approx
\C[t_1,t_2,y_1,\dots,y_d],
\]
with induced action by $(\ol{\mc{W}}(G)\times \C^\times)^\op$ described much
as in \eqref{eq:action-on-XtimesC}, except that we have
\begin{align*}
m\propto (t_1,t_2, y) &=
(t_1,t_2,y+m_1t_1+m_2t_2),
& m=(m_1,m_2)\in \Hom(\Z^2,\Z^d)\approx (\Z^d)^2,
\end{align*}
where $y=(y_1,\dots,y_d)$.  Geometrically, this gives
\[
GL_2(\Z)\rtimes \Z^{d\times 2} \times \C^\times\; \curvearrowright\;
\mc{X}_{U(1)^d}=\mc{X}\times \C^d =\set{(t_1,t_2,y_1,\dots,y_d)\in
  \C^2\times \C^d}{\R
  t_1+\R t_2=\C},
\]
whence $\mc{M}_{U(1)^d}\ra \mc{M}_e$ describes the  the $d$-fold fiber product
of $\mc{E}$ over $\mc{M}$.

\subsection{Remarks on the relation to double loop groups}
\label{subsec:remarks-double-loop}

The construction we just described seems to be a variant of one
described in \cite{etingof-frenkel-cent-ext-current-groups-2d}.   Here
we will briefly describe how to relate the two.

Fix a compact and simply connected Lie group $G$, with maximal torus $T$ and Weyl
group $W$.  By
analogy with loop groups, one has the \emph{double
  loop group}
\[
LLG:= \Map(\T^2,G)
\]
(where now we consider smooth maps), and also the \emph{extended
  double loop group}
\[
LL^\ext G := \Map(\T^2,G)\rtimes \T^2,
\]
where the $\T^2$ acts by rotations.  The group $LL^\ext G$ itself has
an action of $\Aut(\T^2)=GL_2(\Z)$.  This describes a subgroup
$LL^\ext G\rtimes GL_2(\Z)$ of $\mc{W}(G)$.   

The extended double loop group contains a finite dimensional
torus
\[
T^\ext=T^\ext_G:=T\times \T^2,
\]
 where the $T$ corresponds to constant maps $\T^2\ra\{*\}\ra  T\subseteq G$.
The Weyl group of $T^\ext\subset LL^\ext G$ is the
\emph{elliptic Weyl group}
\[
W_{\Ell} = W\ltimes \Hom(\Z^2, \check{T})
\]
of $G$, where $\check{T}=$ the cocharacter lattice of $T$.  The Lie
algebra $\Lie(T^\ext)=\Lie(T\times \T^2)$ inherits an action by $W_{\Ell}$, as well as an
action by $GL_2(\Z)$.

For the trivial group $e$ we have $T^\ext_e=\T^2$, and
$\Lie(\T^2)\otimes\C \approx \C^2$.  Let $\mc{X}_e\subseteq
\Lie(\T^2)\otimes \C$ denote the subset consisting of pairs of
elements in $\C$ which generate a lattice, and define $\mc{X}_G$ as
the preimage with repsect to the evident projection $\pi$:
\[\xymatrix{
{\mc{X}_G} \ar@{>->}[r] \ar[d]
& {\Lie(T\times \T^2)\otimes \C} \ar[d]^{\pi}
\\
{\mc{X}_e} \ar@{>->}[r]
& {\Lie(\T^2)\otimes \C}
}\]
The action by $GL_2(\Z)\ltimes W_\Ell$ restricts to one $\mc{X}_G$,
and acts fiberwise with respect to $\pi$, so that for each $t\in
\mc{X}_e$ we obtain $W_\Ell\curvearrowright \pi^{-1}(t)$.  
When $G=T$ is itself a torus, this  is evidently the same action as
the one we described in the 
previous section, related via the Chern-Weil
isomorphism
\[
\Sym( \Lie(T\times \T^2)^*\otimes \C) \xra{\sim} H^*(B(T\times
\T^2);\C) = H^*(B\mc{W}_0(G);\C).  
\]

Etingof and Frenkel \cite{etingof-frenkel-cent-ext-current-groups-2d}
describe the following construction.  Given a \emph{simply connected}
$G$ with complexification $G_\C$, together with a choice of
holomorphic structure (=complex
structure  + invariant holomorphic 1-form) on $\Sigma$, they describe a ``coadjoint
action'' of $\Map(\Sigma, G_\C)$ on $\Lie(\Map(\Sigma,G_\C))$
(actually a 
twisted version of the usual coadjoint action which depends on the  chosen
holomorphic structure on $\Sigma$).   They show that orbits for
this action correspond to isomorphism classes of holomorphic principal $G$-bundles on
$\Sigma$.  A generic class of orbits are given by the restriction to
the maximal torus: the orbits of  $W_\Ell$ acting on $\Lie(T_\C)$ correspond to
the ``flat and unitary'' holomorphic $G$-bundles on $\Sigma$. 

Examining the formulas in Etingof and Frenkel, one sees that the holomorphic data for
$\Sigma$ corresponds to a choice of point $t\in \mc{X}_e\subset
\Lie(\T^2)\otimes \C$, and that their action $W_\Ell\curvearrowright
\Lie(T_\C)$ coincides with the  action of  $W_\Ell$ on the fiber
$\pi^{-1}(t)\subset \Lie(T\times \T^2)\otimes \C$ that we described
above.   (Note: in the formulation of Etingof and
Frenkel, they do not identify holomorphic structures on $\Sigma$ with points
in $\Lie(\T^2)\otimes \C$; rather, the use the holomorphic structure
to construct a central $\C^\times$-extension of
$\Map(\Sigma,\C^\times)$, so that their coadjoint action is the
natural one on a slice of the Lie algebra of their central extension \cite{etingof-frenkel-cent-ext-current-groups-2d}*{\S3}.)


\section{Looijenga line bundles}
\label{sec:looijenga}

We now describe the main result of this paper: if in our construction
we replace $G=U(1)^d$ with
$\wt{G}=U(1)^d\times_\phi K(\Z,2)=$ a ``central'' extension of $U(1)^d$ by
$K(\Z,2)$, we 
get  the total space of the principal bundle of a Looijenga line
bundle.  We start with the special case of $d=0$, i.e., $\wt{G}=K(\Z,2)$.

\subsection{$\wt{G}=K(\Z,2)$ gives the multiplicative group}

We describe our results in the case that $\wt{G}=K(\Z,2)$.  
We have that
\[
\ol{\mc{W}}(K(\Z,2)) \approx GL_2(\Z)\ltimes \Z
\]
where $GL_2(\Z)$ acts on $\Z=H^2(\Sigma;\Z)$ via the determinant.
We have
\[
H^*(B\mc{W}_0(K(\Z,2));\C) \approx
\C[t_1,t_2,x_1,x_2]/(t_1x_1+t_2x_2),\qquad t_i,x_k\in H^2.
\]
The resulting action 
\[
(\ol{\mc{W}}(K(\Z,2))\times \C^\times)^\op\;\curvearrowright\;
H^*(B\mc{W}_0(K(\Z,2));\C) 
\]
is described by 
\begin{align}\label{eq:action-for-KZ3}
\begin{aligned}
n \propto (t_1,t_2,x_1,x_2) &= (t_1,t_2,x_1-nt_2, x_2+nt_1)
\\
A\propto (t_1,t_2,x_1,x_2) &= (at_1+bt_2, ct_1+dt_2,
\frac{dx_1-cx_2}{\det A}, \frac{-bx_1+ax_2}{\det A}), 
\\
\lambda\propto (t_1,t_2,x_1,x_2) &= ( \lambda t_1, \lambda
t_2,\lambda x_1, \lambda x_2),
\end{aligned}
&&
\begin{aligned} 
n\in \Z,,
\\
A\in GL_2(\Z),
\\
\lambda\in \C^\times.
\end{aligned}
\end{align}
The associated geometric object is
\[
\mc{X}_{K(\Z,2)} = \set{(t,x)\in\C^2}{t_1x_2+t_2x_2=0,\; \R t_1+\R t_2=\C}
\subset \mc{X}\times \C^2.
\]
The projection $\mc{X}_{K(\Z,2)}\ra \mc{X}$ is  a trivial line bundle over
$\mc{X}$, via the nowhere vanishing section
$(t_1,t_2)\mapsto (t_1,t_2,-t_2,t_1)$.  The action 
$\Z\curvearrowright \mc{X}_{K(\Z,2)}$ is fiber-by-fiber, via
translation along this 
section, and so acts freely.  Thus 
\[
\Z\backslash\!\backslash \mc{X}_{K(\Z,2)} \approx \Z\backslash
\mc{X}_{K(\Z,2)} \approx \mc{X}\times \C^\times.
\]
Explicitly, $\Z\backslash \mc{X}_{K(\Z,2)} \xra{\sim} \mc{X}\times
\C^\times$ is given by 
\[
(t_1,t_2,x_1,x_2) \mapsto (t_1,t_2, e^{2\pi i(x_1/t_2)}).
\]
The $GL_2(\Z)\times \C^\times$ action descends to an action on $\mc{X}\times
\C^\times$ of the form
\[
(A,\lambda)\propto (t_1,t_2,u) = (\lambda(at_1+bt_2),
\lambda(ct_1+dt_2),  u^{1/\det A}).
\]
Since elements $A\in GL_2(\Z)$ with $\det A=-1$ switch the two
components of $\mc{X}$, we see that
\[
\mc{M}_{K(\Z,2)} \approx \mc{M}\times \C^\times.
\]
This  is naturally a group object over $\mc{M}$, via the group structure
on $K(\Z,2)$.

\begin{rem}
This will follow from the general theorem \eqref{thm:general-thm}.  To see how it 
arises, consider the Serre spectral sequence for  
$B\Map_0(\Sigma, K(\Z,2))\ra B\mc{W}_0(K(\Z,2))\ra
B\Diff_0(\Sigma)$, which has $E_2^{p,q}=\C[t_1,t_2]\otimes
\C[x_1,x_2,\epsilon]$, with $\len{\epsilon}=(0,3)$.  The only differential
is $d_2(\epsilon)=\pm(t_1x_1+t_2x_2)$.  The terms $(\dots,
\cdots-nt_2, \cdots +nt_1)$ in the first line of \eqref{eq:action-for-KZ3}
ultimately derive from the non-degenerate pairing $H_1\T^2\xra{\sim} H^1\T^2$
adjoint to the Pontryagin product on $H_*\T^2$.
\end{rem}

\subsection{Quadratic functions}
\label{subsec:quadratic-functions}

Let $B$ and $C$ be finitely generated free abelian groups.  A
\dfn{quadratic function} $\phi\colon B\ra C$ is a function such that
\begin{itemize}
\item $\beta(b,b'):= \phi(b+b')-\phi(b)-\phi(b')$ is bilinear, and 
\item $\phi(nb)=n^2\phi(b)$ for $n\in \Z$.
\end{itemize}
The symmetric bilinear form $\beta\colon B\otimes B\ra C$ is called
the \dfn{Hessian form} of $\phi$.  Note that
$\phi(b)=\tfrac{1}{2}\beta(b,b)$, so $\beta$ determines $\phi$.

Let $\Gamma_2B$ be the second degree part of the divided power algebra
on $B$; since $B$ is 2-torsion free, $\Gamma_2B\approx (B\otimes
B)^{\Sigma_2}$.  The function $\gamma_2\colon B\ra \Gamma_2B$ given by
$b\mapsto b\otimes b$ is the universal quadratic function out of $B$,
so that
\[
\Hom(\Gamma_2B, C) \xra[\sim]{\wt\phi\mapsto \wt\phi\circ \gamma_2}
\{\text{quadratic $B\xra{\phi} C$}\}
\]
is a bijection.  We will use the notation $\wt\phi$ for the
homomorphism associated to a quadratic function $\phi$.

A \dfn{bilinear extension} of $\phi$ is any bilinear (but not
necessarily symmetric) map $\omega\colon B\times B\ra C$ such that
$\phi(b)=\omega(b,b)$.  Such extensions always exist (because the
exact sequence $0\ra \Gamma_2 B\ra B\otimes B\ra \Lambda^2B\ra 0$
splits), and any two such extensions differ by an alternating form.  

In terms of a choice of coordinates $B\approx \Z^d$, we have
\begin{equation}\label{eq:quad-func-coords}
\phi(y)= \tfrac{1}{2}\sum_{i,j} c_{ij}y_iy_j, \qquad
\beta(y,y')=\sum_{i,j} c_{ij}y_iy_j', \qquad \omega(y,y')=\sum_{i,j}
d_{ij}y_iy_j', 
\end{equation}
where $(c_{ij})$ is a symmetric integer matrix with $c_{ii}\in 2\Z$,
and $(d_{ij})$ any integer matrix such that $c_{ij}=d_{ij}+d_{ji}$.

\subsection{Case of $\wt{G}=$ extension of $U(1)^d$ by $K(\Z,2)$}
\label{subsec:case-of-f}

Given a topological group $G$ and a map $\wt\phi\colon BG\ra K(\Z,4)$, we
have a fibration sequence of the form
\[
B\wt{G}\ra BG \xra{\wt\phi} K(\Z,4).
\]
We define $\wt{G}$ to be the (based) loop space of the fiber
$B\wt{G}$, modelled as a topological group.  We call this $\wt{G}$ the
\dfn{$K(\Z,2)$-central extension of $G$} corresponding to $\phi$
(though as realized above the extension might not be central).  

Given $G=U(1)^d$, set  $B:= \pi_1 G = H_2(BG,\Z)=\Z^d$, so that
up to homotopy maps $\wt\phi$ correspond to elements 
\[
\wt\phi\in H^4(BU(1)^d;\Z) \approx \mathrm{Sym}^2H^2(BU(1)^d;\Z)
\approx \Hom(\Gamma_2B,\Z),
\]
and thus to quadratic functions $\phi\colon B\ra \Z$.

\begin{thm}\label{thm:main-thm}
Let $\wt{G}$ be a $K(\Z,2)$-central extension of $G=U(1)^d$ associated
to a quadratic function $\phi$ with Hessian form $\beta$, and choose a
bilinear extension $\omega$ of $\phi$.  
\begin{enumerate}
\item We have
\[
\ol{\mc{W}}(\wt{G}) \approx GL_2(\Z)\ltimes E,
\]
where $E$ is a central extension
\begin{equation}\label{eq:cent-ext-F}
0\ra \Z\ra E\ra \Hom(\Z^2,\Z^d)\ra 0,
\end{equation}
defined so that the group law on  $E=\Hom(\Z^2,\Z^d)\times \Z$ takes
the form
\[
(m_1,m_2,n)\cdot (m_1',m_2',n') = \Bigl(m_1+m_1',\; m_2+m_2',\;
n+n'+\bigl(\omega(m_1,m_2')-\omega(m_2,m_1')\bigr)\Bigr) 
\]
where $n,n'\in \Z$ and $m, m'\in 
\Hom(\Z^2,\Z^d)\approx (\Z^d)^2$.

The group $GL_2(\Z)$ acts on $E$ (from the left) by 
\[
A\propto (m_1,m_2,n) = (\frac{dm_1-cm_2}{\det A},\;
\frac{-bm_1+am_2}{\det A},\;n).
\]
\item We have
\[
H^*(B\mc{W}_0(\wt{G});\C) \approx \C[t_1,t_2, y_1,\dots,y_d,
x_1,x_2]/(\phi(y)+ (t_1x_1+t_2x_2)), \qquad t_i,y_j,x_k\in H^2.
\]
\item The action $\ol{\mc{W}}(\wt{G})^\op\;\curvearrowright \;
  H^*(B\mc{W}_0(\wt{G});\C)$ is given (in terms of the
  description in (1)) by
\begin{align}\label{eq:action-for-F}
\begin{aligned}
   n\propto (t,y,x) &= (t_1,t_2,\; y,\; x_1-nt_2, x_2+nt_1),
\\
  m\propto (t,y,x) &= (t,\; y+mt,\; x-\beta(y,m)-\omega(mt,m))
\\ 
A\propto (t,y,x)& = (at_1+bt_2, ct_1+dt_2,\; y,\; \frac{dx_1-cx_2}{\det
  A}, \frac{-bx_1+ax_2}{\det A}), 
\end{aligned}
&&
  \begin{aligned}
    n\in \Z,
  \\
   m\in \Hom(\Z^2,\Z^d),
 \\
  A\in GL_2(\Z)
  \end{aligned}
\end{align}
where $t=(t_1,t_2)$, $y=(y_1,\dots,y_d)$, $x=(x_1,x_2)$.
\end{enumerate}
\end{thm}

\begin{rem}
The second line of \eqref{eq:action-for-F} is in 
compressed form.  In 
full it means
\begin{equation}\label{eq:action-for-F-full}
\begin{aligned}
  (m_1,m_2)\propto (t,y,x) &= (t_1,t_2,\; y+m_1t_1+m_2t_2, 
\\ &\qquad
  x_1-\beta(y,m_1)-\omega(m_1t_1+m_2t_2, m_1), 
\\
  & \qquad x_2-\beta(y,m_2)-\omega(m_1t_1+m_2t_2, m_2)),
\end{aligned}
\end{equation}
where $m=(m_1,m_2)\in \Hom(\Z^2,\Z^d)\approx (\Z^d)^2$. 
\end{rem}

\begin{rem}
Up to isomorphism, the central extension \eqref{eq:cent-ext-F}
depends only on the 
antisymmetrization of the 2-cocycle
$\gamma(m,m')=\omega(m_1,m_2')-\omega(m_2,m_1')$, which is
$\gamma_{\mathrm{antisym}}(m,m')= \gamma(m,m')-\gamma(m',m)=
\beta(m_1,m_2')-\beta(m_2,m_1')$,
and thus depends only on $\phi$, not on $\omega$. 
\end{rem}

The corresponding geometric object $\mc{X}_{\wt{G}}\subseteq
\mc{X}\times\C^d\times \C^2$ is the locus of
$t_1x_1+x_2t_2=-\phi(y)$, subject to $\R t_1+\R t_2=\C$.  The free quotient
$\Z\backslash \mc{X}_{\wt{G}}$ is a 
principal $\C^\times$-bundle over $\mc{X}\times \C^d$.  Thus,
$\mc{M}_{\wt{G}}$ is the total space of a principal $\C^\times$-bundle 
over $\mc{E}^d$.  

In fact, let us consider the quotient of $\mc{X}_{\wt{G}}$ under the
free action 
by $\Z\times
\C^\times$.  Explicitly,
\[
(\Z\times \C^\times) \backslash \mc{X}_{\wt{G}}\xra{\sim}
(\C\smallsetminus \R)\times 
\C^d\times \C^\times
\]
is given by 
\[
(t_1,t_2,\;y_1,\dots,y_d,\;x_1,x_2) \mapsto (\tfrac{t_1}{t_2},\;
\tfrac{y_1}{t_2},\dots, 
\tfrac{y_d}{t_2},\;  e^{2\pi i(x_1/t_2)}) = (\tau,\; z_1,\dots, z_d,\;
u).  
\]
The $\ol{\mc{W}}(G)$-action descends to an action by
$\Hom(\Z^2,\Z^d)\rtimes 
GL_2(\Z)$ on the quotient, given by 
\begin{align}
  \begin{aligned}
    m\propto (\tau, z, u) &= (\tau,\; z+m_1\tau+m_2,\; u e^{2\pi
      i[ -\beta(z,m_1)-\phi(m_1)\tau]}),
\\
  A \propto (\tau,z,u) &= (A\tau,\; (c\tau+d)^{-1}z,\;
  u^{1/\det A} e^{2\pi i(1/\det A)[
    c(c\tau+d)^{-1}\phi(z)]}), 
  \end{aligned}
&&
   \begin{aligned}
     m\in \Hom(\Z^2,\Z^d), \\ A\in GL_2(\Z),
   \end{aligned}
\end{align}
where $A\tau=(a\tau+b)/(c\tau +d)$ and $z=(z_1,\dots,z_d)$.  This describes  the principal
$\C^\times$-bundle over $\mc{E}^{\times d}$ whose associated line
bundle has as sections 
$\theta(\tau,z)$ such that (for $\Im(\tau)>0$ and $A\in
SL_2(\Z)$)
\begin{align*}
  \theta(\tau,\; z+m_1\tau+m_2)& = \theta(\tau,z)\,e^{2\pi
    i[-\beta(z,m_1)-\phi(m_1)\tau]},
\\
  \theta(A\tau,\; (c\tau+d)^{-1}z)&=\theta(\tau,z)\, e^{2\pi
    i[c(c\tau+d)^{-1}\phi(z)]}. 
\end{align*}
In other words, we obtain the \emph{Looijenga line bundle} associated to the
quadratic form $\phi$ \cite{looijenga-root-systems}.  

\begin{rem}
Suppose $\phi\colon B=\Z^d\ra \Z$ is a non-degenerate quadratic
function.  Then with our
conventions, the  line bundle $L_\phi$ associated to $\phi$, admits a
non-trivial \emph{holomorphic} section over $C_\tau^d$ (for any chosen
$C_\tau:=\C/(\Z\tau+\Z)$ with $\Im(\tau)>0$) if and only if 
$\phi$ is \emph{positive definite}.  The main example of interest is 
the positive definite quadratic function $\phi$  associated to the
Killing form on the coroot lattice of a simply connected compact Lie
group; in this case $\phi$ is 
invariant under the action of the Weyl group, so the bundle
$L_\phi$ is equivariant for the Weyl group.

To see the existence of sections in this case, we use
\cite{mumford-abvar}*{I.2 and I.3}.  In the 
notation of \cite{mumford-abvar}*{I.2, p.\ 15--16; I.3, p.\ 24--25},
the line bundle 
$L_\phi|C_\tau^d$ is 
described by a 1-cocycle $e$ on $U=\Z^d\tau+\Z^d\subseteq \C^d$ with
coefficients in holomorphic functions $\C^d\ra \C^\times$, given by
\[
e_u(z) = e^{2\pi i f_u(z)},\qquad
f_{m_1\tau+m_2}(z)=-\beta(z,m_1)-\tfrac{1}{2}\beta(m_1,m_1)\tau,
\qquad m_1,m_2\in \Z^d.
\]
By \cite{mumford-abvar}*{I.2, p.\ 18, Proposition}, 
\[
E(u,u') := f_{u'}(z+u) + f_u(z) -f_u(z+u') -f_{u'}(z), \quad \text{any
  $z\in \C^d$,}
\]
defines an alternating 2-form $E\colon U\times U\ra \Z$ which
represents the Chern class of $L_\phi|C_\tau^d$.  We calculate that in our
case, 
\[
E(m_1\tau+m_2, m_1'\tau+m_2') = \beta(m_1,m_2')-\beta(m_2,m_1').
\]
Extend $E$ to an $\R$-linear form $\C^d\times \C^d\ra \R$ and set
$H(x,y):= E(ix,y)+ iE(x,y)$.  Then $H$ is a Hermitian form with $\Im
H=E$.  By (\cite{mumford-abvar}*{I.3, p.\ 26}, proposition and preceding
discussion), if $H$ is non-degenerate, then $L_\phi|C_\tau^d$ admits non-zero
holomorphic sections if and only if $H$ is positive definite, in which
case $\dim H^0(C_\tau^d, L_\phi|C_\tau^d) = \sqrt{\det E}$ (express $E$ as a
matrix using a $\Z$-basis of $U$).  
We calculate that in our case, 
\[
H(x,x) = (\Im\tau)^{-1}\beta(x,\ol{x}), \qquad x\in \C^d. 
\]
As $\Im \tau>0$ and $\beta(x,y)=\sum c_{ij}x_iy_j$ is a symmetric form
on $\C^d$ with $c_{ij}\in
\Z\subseteq \R$,
we see that $H$ is non-degenerate/positive definite on $\C^d$ if and only if
$\beta$ is non-degenerate/positive definite on $\R^d$, and if so we 
have 
$\sqrt{\det E} = \det (c_{ij})$.  
\end{rem}

\begin{rem}
For $\phi\colon B\approx \Z^d\ra \Z$ positive definite,
sections $\theta_u$ of $L_\phi|C_\tau^d$ 
are given by
$\theta_u(\tau,z) = \sum_{v\in B} e^{2\pi i [
  -\beta(z,u+v)+\phi(u+v)\tau ]}$ for 
$u\in B\otimes \R$ such that $\beta(u,B)\subseteq \Z$,
\cite{looijenga-root-systems}*{\S4}.   
\end{rem}

\subsection{Proof of the theorem}

We will derive \eqref{thm:main-thm} from a more general (and
coordinate invariant) statement \eqref{thm:general-thm}, whose setup
and proof takes up
\S\S\ref{sec:spaces-in-2-3}--\ref{sec:conventions}.  It  
is entirely calculational, and amounts to completely describing the
homotopy type of the spaces  $B\mc{W}(\wt{G})$.  In particular, the
key is to compute all Whitehead products in the homotopy groups of
this space. 

 We note that one can instead regard $\wt{G}$ as arising from  a
\emph{Lie $2$-group}, specifically as a 2-group 
extension as considered in \cite{ganter-categorical-tori}.  It seems
likely that  2-group methods should lead to a more 
informative proof of the results shown here.

\section{Isogenies}
\label{sec:isogenies}

We describe how, according to the picture of the previous sections,
finite 
coverings of genus 1 surfaces 
correspond to isogenies of elliptic curves.

Fix a finite covering  map
$f\colon 
\Sigma'\ra \Sigma$ between two surfaces.   Let $\Diff(f)\subset
\Diff(\Sigma)\times \Diff(\Sigma')$ denote the group of pairs of
diffeomorphisms compatible with $f$.  We note that the projection map
$\Diff(f)\xra{t} \Diff(\Sigma)$ is a finite covering map, while  the projection map
$\Diff(f)\xra{s} \Diff(\Sigma')$ is injective and induces a homotopy equivalence
between $\Diff(f)$ and a union of path components of $\Diff(\Sigma')$,
corresponding to a finite index subgroup of $\pi_0\Diff(\Sigma')$.  

Given any group $G$, we can form a diagram as follows
\begin{equation}\label{eq:isogeny-diagram}
\xymatrix{
{B\mc{W}^{\Sigma}(G)} \ar[d]
& {(Bt)^*B\mc{W}^{\Sigma}(G)} \ar[l] \ar[dr] \ar@/^2ex/[rr]^{f^*}
&& {(Bs)^*B\mc{W}^{\Sigma'}(G)} \ar[r] \ar[dl]
& {B\mc{W}^{\Sigma'}(G)} \ar[d]
\\
{B\Diff(\Sigma)} 
&& {B\Diff(f)} \ar[ll]^{Bt} \ar[rr]_{Bs}
&& {B\Diff(\Sigma')}
}\end{equation}
where the trapezoids are homotopy pullbacks.  That is,
$(Bt)^*B\mc{W}^\Sigma(G)
\approx B(\Map(\Sigma,G)\rtimes\Diff(f))$ and
$(Bs)^*B\mc{W}^{\Sigma'}(G)\approx B(\Map(\Sigma',G)\rtimes\Diff_+(f))$,
while  the map labeled $f^*$ 
is obtained from the map $\Map(\Sigma, G)\ra \Map(\Sigma',G)$ given
by restriction along $f$.

The observation is that, after applying the construction of
\S\ref{subsec:geometric-picture},  the map $f^*$ presents an isogeny
of curves, 
of degree equal to the degree of $f$.  To see this, we consider an
explicit example.

\begin{exam}
Fix $\Sigma=\Sigma'=\T^2$, and let $f\colon \Sigma'\ra \Sigma$ be the
map induced by  left 
multiplication by some integer matrix $B$.
Set $\Gamma_B:= GL_2(\Z)\cap B^{-1}\,GL_2(\Z)\, B$.   Note using a
suitable choice of bases of $H_1\Sigma$ and $H_1\Sigma'$, the matrix
$B$ can be given  the form  $B=\left(\begin{smallmatrix} M&0\\
0&MN\end{smallmatrix}\right)$ for some $M,N\geq 1$, in which case $\Gamma_B
= \Gamma_0(N) = \set{\left(\begin{smallmatrix}
      a&b\\c&d\end{smallmatrix}\right)\in GL_2(\Z)}{c\equiv 0\mod
  N}$.

Then there is a weak equivalence of topological groups
\[
\Gamma_B\ltimes \T^2 \xra{\sim} \Diff(f),
\]
so that the projections  $\Diff(\Sigma) \la \Diff(f) \ra \Diff(\Sigma')$
correspond to
\[
GL_2(\Z)\ltimes \T^2 \xla{(BAB^{-1},Bt)\mapsfrom (A,t)} \Gamma_B\ltimes \T^2
\xra{(A,t)\mapsto(A,t)} GL_2(\Z)\ltimes \T^2.
\]
Let $G=U(1)$, form $\Spec_\an$ of the cohomology of universal covers
of objects in \eqref{eq:isogeny-diagram}, and restrict to the subset
$\mc{X}=\set{(t_1,t_2)}{\R t_1+\R t_2=\C}$.  Together with actions of
fundamental groups and the grading action by $\C^\times$, the 
middle
triangle of \eqref{eq:isogeny-diagram} is seen to have the form
\[\xymatrix{
{(\Gamma_B\ltimes' \Z^2)\times \C^\times \bbss
  \mc{X}\times \C}  
\ar[dr] \ar[rr]^{f^*} 
&& {(\Gamma_B\ltimes \Z^2)\times \C^\times \bbss
  \mc{X}\times \C} \ar[dl]
\\
& {\Gamma_B\times \C^\times\bbss \mc{X}} 
}\]
where $f^*$ is induced by the identity on $\mc{X}\times \C$, and both
maps $\mc{X}\times \C\ra \mc{X}$ are the evident projection.  The
semidirect product $\Gamma_B\ltimes \Z^2$ is induced by the
tautological action $\Gamma_B\subset GL_2(\Z)$, while the semidirect
product $\Gamma_B\ltimes'\Z^2$ is induced by the homomorphism
$A\mapsto BAB^{-1}\colon \Gamma_B\ra GL_2(\Z)$.  The action in the
upper-right corner is
\[
A\propto (t,y)=(At,y), \quad m\propto (t,y)=(t,y+mt), \quad
\lambda\propto(t,y)=(\lambda t,\lambda y),
\]
while the action in the upper-left corner is
\[
A\propto (t,y)=(At,y), \quad m\propto(t,y)=(t,y+mBt), \quad
\lambda\propto (t,y)=(\lambda t, \lambda y)
\]
where $A\in \Gamma_B$, $m\in \Z^2$ (treated as a row vector), $t\in
\mc{X}$ (treated as a column vector),  $y\in \C$, and $\lambda\in
\C^\times$.

Thus, in the ``fibers'' over $(t_1,t_2)\in \mc{X}$ we obtain (after
taking quotients by $\Z^2$-actions) the projection
$\C/\bigl((Bt)_1\Z+(Bt)_2\Z\bigr) \ra \C/\bigl(t_1\Z+t_2\Z\bigr)$, an
isogeny of degree $\det B$.   E.g., for $B=\left(\begin{smallmatrix} M&0\\
0&MN\end{smallmatrix}\right)$ we get $\C/\bigl( Mt_1\Z +
MNt_2\Z\bigr)\ra \C/(t_1\Z+t_2\Z)$.  
\end{exam}

\section{Remarks on the formalism}
\label{sec:remarks-on-formalism}

\subsection{Remarks on the construction of an equivariant cohomology
  theory}

We can easily produce for each group $G$ that we consider  an
equivariant cohomology theory of the form
\[
E^*_G\colon \bigl(\text{$G$-CW-complexes}\bigr)^\op \ra
\bigl(\text{$\ol{\mc{W}}(G)$-equivariant
  $H^*(B\mc{W}_0(G);\C)$-algebras}\bigr). 
\]
Given a $G$-space $X$ let 
\[
\Map^\gh_G(\Sigma\times G, X) \subseteq \Map_G(\Sigma\times G,X)
\]
be the subspace consisting of \dfn{ghost maps}, i.e., $G$-equivariant
maps $f\colon \Sigma\times G \ra X$ such that $f(\Sigma\times G)$ is
contained in a single $G$-orbit.  The ghost maps are invariant under the evident action of
$\mc{W}(G)$ on $\Map_G(G\times \Sigma, X)$, so we can define
\[
E^*_G(X) := \bigl( \ol{\mc{W}}(G)^\op\;\curvearrowright \;
H^*(\Map^\gh_G(\Sigma\times G,X)_{h\mc{W}_0(G)}; \C)\bigr).   
\]
That this is a cohomology theory amounts to the observations that  (i)
$X\mapsto \Map^\gh_G(\Sigma\times G, X)$  preserves pushouts along
cofibrations and (ii) $\Map_G^\gh(\Sigma\times G, T\times X)\approx
T\times 
\Map_G^\gh(\Sigma\times G, X)$ when $T$ has 
trivial $G$-action.  
\begin{exam}
Let $G=U(1)$ and $X=U(1)/\mu_N$.  Then 
\[
E^*_{U(1)}(U(1)/\mu_N) \approx \prod_{(n_1,n_2)\in
  \Z^2}\C[t_1,t_2,y]/(y-(n_1/N)t_1-(n_2/N)t_2),
\]
which is an algebra over $H^*(B\mc{W}_0(U(1));\C)\approx
\C[t_1,t_2,y]$ in the 
obvious way and which carries an evident compatible action by $\ol{\mc{W}}(U(1))=
GL_2(\Z)\ltimes \Z^2$.  
\end{exam}
Ideally one would like to ``analytify'' the equivariant module $E^*_G(X)$,
to obtain a sheaf of $\mc{O}_{\mc{M}_G}$-algebras on $\mc{M}_G$, to be coherent at least if $X$
is a finite $G$-CW-complex;  we would  then hope to take it as a model for
Grojnowski's equivariant elliptic cohomology.  Unfortunately, the most obvious way to do
this (e.g., by tensoring up from algebraic to holomorphic functions),
though exact, behaves poorly  on most $E^*_G(X)$
(which are often non-Noetherian, even when $X$ is a $G$-orbit).

\subsection{Remarks on derived constructions}

In this paper we have been content to produce examples of ``classical''
geometric objects, e.g., 
complex analytic spaces.  However, we know that elliptic cohomology
wants to take values in sheaves on a \emph{derived} geometric object,
along the lines of \cite{lurie-elliptic-survey}. I don't know how to
make such a derived construction; however, I'll give some speculation here.

Fix a commutative dga\footnote{We use homological grading here, so
  $x\in C^q$ has $\len{x}=-q$.}
$\C[u^\pm]$, where $\len{u}=2$ with $du=0$.  This admits an
evident grading coaction by the Hopf algebra $\C[\lambda^\pm]$ with
$\len{\lambda}=0$ and $d(\lambda)=0$, by $u\mapsto \lambda\otimes u$.  Thus
$\mathbb{G}_m:=\Spec^\der \C[\lambda^\pm]$ acts on $\Spec^\der
\C[u^\pm]$.

For a space $X$, let $C^*X$ denote a functorial commutative dga model
for the cochains on $X$ with $\C$ coefficients; e.g., we could take
$C^*X$ to be the PL-de Rham forms on $X$.  
Thus $\Spec^\der C^*X\otimes_\C \C[u^\pm]$ inherits an action by
$\mathbb{G}_m$.  We can 
then plug in $\ol{W}(G)\curvearrowright B\mc{W}_0(G)$ as above to
obtain 
\[
\ol{W}(G)\times \mathbb{G}_m \curvearrowright \Spec^\der
C^*B\mc{W}_0(G)\otimes_\C \C[u^\pm],
\]
a derived scheme equipped with an action by a group scheme.

At this point we  posit the existence of a \emph{derived
  analytification} functor $\Spec^\der_\an$, which takes as input a
commutative dga $A^*$ over $\C$, as gives as output a derived complex
analytic space $\Spec^\der_\an A^*=(X,\mc{O})$, in some suitable
$\infty$-category $\An^\der$ of derived analytic spaces.  It should
have the property that 
(at least in 
the examples we care about), the underlying complex analytic space
$(X,H^0(\mc{O}))$ is equivalent to $[\Spec H^0A^*]_\an$.
Given this, we would could then proceed to construct derived versions
of $\mc{X}_G$ and $\mc{M}_G$ as desired.

An $\infty$-category $\An^\der$ has been constructed in work of
Lurie \cite{lurie-dag9} and Porto
\cite{porta-dcag-1}, and in fact comes equipped with an
analytification functor.  A significant issue in carrying out this
program (as pointed out to me by Mauro Porto) is that the ``rings''
which appear in this model are fundamentally $(-1)$-connected objects,
whereas the rings we want to consider are naturally non-connected, and
in fact are generally 2-periodic.

\subsection{Remarks on functoriality}

As we have described it, our construction $G\mapsto \mc{M}_G$ is
functorial with respect to homomorphisms of groups.  For a derived
version of this construction it is highly desirable to have an
enhanced ``stacky''  version of this functoriality, where homomorphisms are
enriched to maps between classifying spaces (not necessarily basepoint
preserving), i.e., we should have
\[
\Map_{\mr{Top}}(BG, BG') \ra \Map_{\An^\der}(\mc{M}^\der_G,
\mc{M}^\der_{G'}).  
\]
This extneded functoriality should apply not just to tori but to
$K(\Z,2)$-extensions of them, and thus should be consistent with
Lurie's notion of \emph{2-equivariance}
\cite{lurie-elliptic-survey}*{\S5}.  
I'll briefly indicate how to achieve this; it may be 
enlightening even in the non-derived case. 

Consider $\Map(\Sigma, BG)$, the space parameterizing principal
$G$-bundles over $\Sigma$.  There is a distinguished path
component $\Map_0(\Sigma,BG)\subseteq \Map(\Sigma,BG)$ corresponding
to trivializable bundles, which is  equivalent to $B\Map(\Sigma, G)$.
There is a corresponding
path component
\[
\mc{P}(BG) \subseteq \Map(\Sigma,BG)_{h\Diff(\Sigma)}
\]
equivalent to $B\mc{W}(G)$.  Note that a map $BG\ra BG'$
sends $\mc{P}(BG)\ra \mc{P}(BG')$.  

Thus, ``enhanced functoriality'' follows once we describe how to
functorially obtain a 
derived stack from a suitable connected space $X$ (such as
$X=\mc{P}(BG)$).   

For each path connected space $X$, make an arbitrary choice of
universal cover $p\colon \wt{X}\ra X$, and write $G$ for its group of deck
transformations.  Note that $p$ is a principal $G$-bundle.  We get a
topological quotient stack $G\bbss \wt{X}$ which is 
equivalent to $X$.  After taking cohomology we obtain a total quotient
stack $G\bbss \Spec H^*\wt{X}$;
replacing cohomology with cochains gives the corresponding derived
object.  When $X=\mc{P}(BG)$ this recovers the construction
$\ol{W}(G)\bbss H^*B\mc{W}_0(G)$.  

Consider a  map $f\colon X\ra Y$ to another path connected space, we
and write $(\wt{Y},q,H)$ for the 
analogous  choices for $Y$.  Then $f$ induces a map $G\bbss \Spec
H^*\wt{X}\ra H\bbss 
\Spec H^*\wt{Y}$ of stacks which is represented by a \emph{bibundle},
as follows.  Consider
\[
\wt{X} \xla{\pi} \Lift(f)\times \wt{X} \xra{\epsilon} \wt{Y}
\]
where $\Lift(f)=\set{\wt{f}\colon \wt{X}\ra \wt{Y}}{q\wt{f}=fp}$ is
the set of lifts of $f$ to the universal covers, $\pi$ is the
projection map, and $\epsilon$ is the  evaluation map.  We have
\begin{itemize}
\item $G$ acts on $\Lift(f)\times \wt{X}$ from the left by $g\cdot
  (\wt{f},\wt{x}) = (\wt{f}g^{-1}, g\wt{x})$, 
\item $H$ acts on $\Lift(f)\times \wt{X}$ from the right by
  $(\wt{f},\wt{x})\cdot h = (h^{-1}\wt{f},\wt{x})$,
\item the group actions on $\Lift(f)\times \wt{X}$ commute,
\item $\pi$ is equivariant with respect to $G$ and $H$ (where $H$ acts
  trivially on $\wt{X}$),
\item $\epsilon$ is equivariant with respect to $G$ and $H$ (where $G$
  acts trivially on $\wt{Y}$), and
\item $\pi$ describes a \emph{$G$-equivariant principal $H$-bundle}
  over $\wt{X}$.
\end{itemize}
That is, the diagram describes a \emph{bibundle} from the topological
groupoid $G\bbss \wt{X}$ to $H\bbss \wt{Y}$; i.e, it is a ``stacky''
presentation of $f$ in terms of the chosen covers.

Taking cohomology (or in the derived context, cochains) gives 
\[
\Spec H^*\wt{X} \xla{\pi} \Lift(f)\times \Spec H^*\wt{X}
\xra{\epsilon} \Spec H^*\wt{Y},
\]
exhibiting a bibundle between groupoid schemes,  i.e., representing a
map $G\bbss \Spec H^*\wt{X}\ra H\bbss \Spec H^*\wt{Y}$ of stacks.

\begin{exam} 
Applying this to our set-up in the case of a map $f\colon *\approx Be\ra
BU(1)$ gives 
\[
\mc{X} \xla{\pi} (GL_2(\Z)\ltimes\Z^2)\times \mc{X} \xra{\epsilon}
\mc{X}\times \C 
\]
with $\epsilon((B,n),t)=(Bt,nBt)$.  The
groups $G=GL_2(\Z)$ and $H=GL_2(\Z)\rtimes \Z^2$ act on $\Z^2\times
\mc{X}$ by $A\cdot ((B,n),t) = ((BA^{-1}, n), At)$ and $((B,n),t)\cdot
(A',m)= ((A'B, (n+m)A'^{-1}), t)$.  
\end{exam}

\subsection{Remarks on the 1-dimensional case}

We can carry out the analogue of our constructions in the case that
$\Sigma$ is a circle rather than a torus.  The relevant calculations
can be read off from \eqref{thm:general-thm}.  The main differences
are that 
in this case we take
\[
\mc{X}=\mc{X}_e= \set{t\in \C}{t\neq0},
\]
the set of vectors which generate a rank 1 lattice in $\C$.  Then we
easily discover that $\mc{M}_{U(1)}$ is the ``universal multiplicative
group'' living over $\mc{M}_e\approx (\Aut(\mathbb{G}_m)\bbss *)\approx
(\{\pm 1\}\bbss *)$.  The central extension groups turn out to be
invisible from this point of view, since  
$\mc{M}_{K(\Z,2)} \approx \mc{M}_{e}$.  

\section{Some spaces and groups}
\label{sec:spaces-in-2-3}

The spaces $X=B\mc{W}_0(\wt{G})$ that we need to deal with are simply connected
3-types such that $\pi_2$ and $\pi_3$ are finitely generated and
free.  We first discuss some general facts and conventions about such
spaces, concluding with the calculation of $H^*(X;\Q)$ in terms
of the Whitehead product in $\pi_*X$ in good cases; all of this material is surely
standard.  We next describe an explicit topological group model for the
central extensions $\wt{G}=U(1)^d\times_\phi K(2,\Z)$ that  we need to consider.

\subsection{Simply connected 3-types with all homotopy groups finitely
  generated and free}
\label{subsec:sc-3-types}

Let $\mc{C}$ denote the full subcategory of spaces $X$ which are
(i) simply connected, (ii) have $\pi_kX\approx 0$ for $k\geq 4$, and (iii) have
$\pi_2X$ and $\pi_3X$ which are finitely generated free abelian
groups.  Write $h\mc{C}$ for the associated homotopy category.

For $X\in \mc{C}$ with $\pi_2X=B$ and $\pi_3X=C$, the 
Whitehead product \eqref{subsec:whitehead-products} $[-,-]\colon
\pi_2X\times 
\pi_2X\ra \pi_3X$ 
defines a bilinear symmetric form 
\[
\beta\colon B\otimes B\ra C.
\]
Precomposition $\circ \eta\colon \pi_2X\ra \pi_3X$ with the Hopf map
$\eta\in \pi_3S^2$ is a function
\[
\phi\colon B\ra C,
\]
quadratic in the sense of \S\ref{subsec:case-of-f}, 
which satisfies 
\[
\phi(y+y') = \phi(y)+\beta(y,y')+\phi(y').
\]
(This identity fixes our preferred choice of generator $\eta$ of
$\pi_3S^2$.)  We call $\phi$ the \dfn{quadratic invariant} of $X$, and
$\beta$ the associated \dfn{Hessian form}.

It is classical that the data of $(B,C,\phi)$ is a complete invariant
for the homotopy type of $X\in \mc{C}$.  In fact, $X\mapsto
\phi$ defines an equivalence between the homotopy category
$h\mc{C}$ of such spaces, and the category of quadratic
functions between finitely generated free groups.  

Let $i\colon K(C,3)\ra X$ and $j\colon X\ra K(B,2)$ be maps, unique up
to homotopy, which induce identity on the relevant homotopy groups.
We can furthermore extend to a fibration sequence
\[
X \xra{j} K(B,2) \xra{\psi} K(C,4),
\]
i.e., so that $j$ is identified with the tautological map from the
homotopy fiber of $\psi$.  

\begin{prop}
\label{prop:whitehead-product-vs-k-invariant}
Let $X\in \mc{C}$ with quadratic invariant $\phi$ and Hessian form
$\beta$, and  consider $b,b'\in B=\pi_2X$.  There exists a homotopy
commutative diagram
\[\xymatrix{
{S^3} \ar[r]^-{w} \ar[d]_{f}
& {S^2\vee S^2} \ar@{>->}[r] \ar[d]_{(b,b')}
& {S^2\times S^2} \ar@{->>}[r] \ar[d]_{g}
& {S^4} \ar[d]_{\wt{f}}
\\
{K(C,3)}  \ar[r]_-{i} 
& {X} \ar[r]_-{j}
& {K(B,2)} \ar[r]_-{\psi}
& {K(C,4)}
}\]
where $w$ is the universal Whitehead product, and $\wt{f}\colon
S^4\approx S^3\sm
S^1\ra K(C,4)$ is adjoint to $f\colon S^3\ra \Omega K(C,4)\approx
K(C,3)$.  Furthermore,
\begin{enumerate}
\item $g_*\colon H_4(S^2\times S^2)\ra H_4K(B,2)\approx \Gamma_2 B
  \subseteq (B\otimes B)^{\Sigma_2}$ sends $[S^2]\times [S^2]\mapsto
  b\otimes b'+b'\otimes b$, 
\item $\wt{f}_*\colon H_4S^4\ra H_4K(C,4)\approx C$ sends
  $[S^4]\mapsto \beta(b,b')$,  and
\item $\psi_*\colon H_4K(B,2)\approx \Gamma_2B \ra H_4K(C,4)\approx C$
coincides with $\wt{\phi}\colon \Gamma_2B\ra C$, the homomorphism
associated to $\phi$ as defined in \eqref{subsec:quadratic-functions}.
\end{enumerate}
\end{prop}
\begin{proof}
We are using the tautological identification $H_4K(B,2)\approx
(B\otimes B)^{\Sigma_2}$ dual to $H^4K(B,2)\approx (B\otimes
B)_{\Sigma_2}$ defined by the cup product.    With respect to this 
identification, the H-space structure on $K(B,2)$ induces a Pontryagin
product $H_2K(B,2)\otimes H_2K(B,2)\ra H_4K(B,2)$ given by
$b\otimes b'\mapsto b\otimes b'+b'\otimes b\colon B\otimes B\ra \Gamma_2(B)$.

The construction of the diagram is straightforward.  In particular,
we can use the H-space structure on $K(B,2)$ to define $g$ as the
composite $S^2\times S^2 \xra{b\times b'} K(B,2)\times K(B,2)\ra
K(B,2)$, from which statement (1) follows immediately.
Statement (2) is immediate from the fact that $f$ and $\wt{f}$ are
adjoint, and that $if=[b,b']\colon S^3\ra X$.  Statement (3) then
follows from the commutativity of the diagram.
\end{proof}

Thus,  any $X\in \mc{C}$ is the homotopy fiber of the characteristic
class in $H^4(K(B,2), C)$ corresponding to  its quadratic invariant. 

\subsection{Rational cohomology ring of $X\in \mc{C}$}

Say that a quadratic function $\phi\colon B\ra C$ is
\dfn{regular} if the function 
\[
\wt{\phi}^* \colon C^*\otimes \Q \ra \Hom(\Gamma_2B,\Q) \approx
\Sym^2(B^*\otimes \Q) 
\]
dual to $\wt{\phi}\colon \Gamma_2B\ra C$ 
sends some basis of $C^*\otimes \Q$ to a regular sequence in the ring
$\Sym(B\otimes \Q)^*$.  

\begin{rem}
If $B=\Z^d$, $C=\Z^e$, and $\phi(y)=(\phi_1(y),\dots, \phi_e(y))$ with
$\phi_k(y)=\tfrac{1}{2}\sum_{i,j}c_{ij}^k y_iy_j$, then $\phi$ is
regular if and only if the sequence of polynomials
$\phi_1(y),\dots,\phi_e(y)$ form a regular sequence in $\Q[y_1,\dots,y_d]$.

In particular,  if $e=1$, then $\phi$ is regular if and only if $\phi\neq0$.  
\end{rem}

\begin{prop}\label{prop:cohomology-23-space}
Let $X\in \mc{C}$ with quadratic invariant $\phi$.  Then the map
$j^*\colon H^*(K(B,2);\Q)\ra H^*(X;\Q)$ factors through
\[
\Sym(B^*\otimes\Q)/(\wt{\phi}^*(C^* \otimes \Q)) \ra H^*(X;\Q)
\]
where $\wt{\phi}^*\colon C^*\ra (\Gamma_2B)^*$ is the $\Z$-dual to
$\wt{\phi}$.  Furthermore, the above map is an isomorphism of rings
when $\phi$ is regular.
\end{prop}
\begin{proof}
The Serre spectral sequence for the fibration sequence $K(C,3)\xra{i}
F\xra{j} K(B,2)$ has  
\[
E_2=E_4=H^*(K(B,2); H^*(K(C,3);\Q)) \approx \Sym(B^*\otimes \Q)\otimes
\Lambda(C^*\otimes \Q).
\]
The first non-trivial differential is $d_4\colon E_4^{0,3}\ra
E_4^{4,0}$, which must be  $\pm\wt\phi^*$ by
\eqref{prop:whitehead-product-vs-k-invariant}.  The 
regularity condition  is what
is needed for $E_5^{p,q}$ with $q>0$ to vanish, so that the spectral
sequence collapses to 
$E_\infty^{*,*}=E_5^{*,0}$.  

\end{proof}

\subsection{An explicit group model for central extensions}
\label{subsec:explicit-group-model}

Every space $X\in \mc{C}$ is equivalent to the classifying space of a
topological group.  We give an explicit construction of such a group
as a central extension.  In particular, given a bilinear map $\omega\colon
B\otimes B\ra C$ between finitely generated free groups, we construct
a topological group $G_\omega$ so that $X=BG_\omega\in \mc{C}$ has
quadratic invariant $\phi$ with $\omega$ as its bilinear extension,
and thus sits in a fiber sequence $K(C,3)\ra X\ra K(B,2)\xra{\wt\phi}
K(C,4)$.   In particular, this produces an explicit model for our
extension groups
$U(1)^d\times_\phi K(\Z,2)$.

Let $K(B,1)_\bullet$ and $K(C,2)_\bullet$ be simplicial abelian
groups, degreewise free, together with identifications $B\approx
\pi_1K(B,1)_\bullet$ and $C\approx \pi_2K(C,2)_\bullet$, and all other
homotopy groups trivial.  There exists a map 
\[
\kappa\colon K(B,1)_\bullet\otimes K(B,1)_\bullet \ra
K(C,2)_\bullet
\]
of simplicial abelian groups
inducing $\omega$ on $\pi_2$, which is unique up to homotopy.  We fix
such a choice of $\kappa$.

Consider the composite map of simplicial sets
\[
K(B,1)_\bullet \times K(B,1)_\bullet \xra{(x,y)\mapsto x\otimes y}
K(B,1)_\bullet\otimes K(B,1)_\bullet\xra{\kappa} K(C,2)_\bullet.
\]
Taking geometric realization produces a map of spaces which we also denote
\[
\kappa\colon K(B,1)\times K(B,1)\ra K(C,2),
\]
which is a bilinear map between topological abelian groups (and so
factors through $K(B,1)\sm K(B,1)$).  Let
 $G_\omega$ be the space $K(B,1)\times
K(C,2)$ with  group law\footnote{The group $G_\omega$ really depends
  on the choice of $\kappa$, but all our computations about it will
  only depend on $\omega$.}
\[
(y,x)\cdot (y',x') := (y+y', -\kappa(y,y')+x+x'), \qquad
(y,x),(y',x')\in G_\omega.
\]
Note that inversion in $G_\omega$ is given by 
\[
(y,x)^{-1} = (-y, -\kappa(y,y)-x),
\]
while the commutator  is given by 
\begin{equation}\label{eq:g-commutator}
  (y,x)\cdot(y',x')\cdot(y,x)^{-1}\cdot (y',x')^{-1} = (0,
  -\kappa(y,y')+\kappa(y',y)).  
\end{equation}
Thus $G_\omega$ is a central extension of $K(B,1)$ by $K(C,2)$,
and we have evident isomorphisms $\pi_1G_\omega\approx B$ and $\pi_2G_\omega\approx
C$.  

The commutator $G_\omega\sm G_\omega\ra G_\omega$ defines the Samelson product
\[
\pairing{-}{-} \colon \pi_p G_\omega\times \pi_q G_\omega\ra \pi_{p+q} G_\omega.
\]
\begin{prop}
The Samelson product
$\pi_1G_\omega\times \pi_1G_\omega\ra 
\pi_2G_\omega$ is given by 
\[
\pairing{b}{b'}=-\omega_\sym(b,b') := -\omega(b,b')-\omega(b',b).
\]
\end{prop}
\begin{proof}
The map $\kappa\colon K(B,1)\sm K(B,1)\ra K(C,2)$ induces $\omega$
on homotopy groups by construction, and therefore $(y,y')\mapsto
\kappa(y',y)$ induces 
$(b,b')\mapsto -\omega(b',b)$ on homotopy groups, with sign introduced
by switching the order of the two classes in $\pi_1G_\omega$.  The result follows 
from \eqref{eq:g-commutator}. 
\end{proof}

\begin{prop}
Let $X=BG_\omega$.  Then the Whitehead product $\pi_2X\times
\pi_2X\ra \pi_3 X$ is given by 
\[
[b,b'] = \omega_\sym(b,b')= \omega(b,b')+\omega(b',b).
\]
\end{prop}
\begin{proof}
This is a special case of the relation between the Whitehead and
Samelson products \eqref{eq:sam-whit-relation}; in this dimension, the
two differ by a sign.
\end{proof}

Thus, the space $X=BG_\omega$ has quadratic invariant  
$\phi\colon B\ra C$, with associated Hessian form
$\beta=\omega_\sym\colon B\otimes B\ra C$.

\section{The main theorem}

In this section we restate our main theorem \eqref{thm:main-thm}, but
in terms of our explicit models for $G$, and in somewhat
more 
generality, in that we allow for a torus of rank other than 2.

\subsection{The group $W^T(G)$}

Fix a finitely generated free abelian group $L$, and let $T:= L\otimes
\T$ be the associated torus.  We consider the semidirect product group
$D(T):= \Aut(T)\ltimes T$ with group law
\[
(A,t)\cdot (A',t') := (AA', (A')^{-1}t+t').
\]
Note that $D(T)$ acts on the space $T$ by 
\[
(A,t)\cdot s = A(s+t).
\]

Given a topological group $G$, we define a group
\[
W^T(G) := \Map(T,G) \rtimes D(T)
\]
with group law given by 
\[
(g,f)\cdot (g',f') = (\bigl(s\mapsto g(s)\cdot g'(f^{-1}(s))\bigr), ff' ).
\]
We write $W^T_0(G)\subseteq W^T(G)$ for the identity component, and
$\ol{W}^T(G):= W^T(G)/W^T_0(G)$ for the quotient.  

Given $\omega\colon B\otimes B\ra C$, we will compute the homotopy
type of the classifying space $BW_0^T(G_\omega)$, together with the
evident action of $\ol{W}^T(G_\omega)$ on its homotopy groups.  

\subsection{Homology and cohomology of $T$}
\label{subsec:torus-conventions}

Because $T$ is an abelian group, $H_*T$ is naturally a graded
commutative Hopf algebra.  The iterated coproduct $\psi\colon H_pT\ra
H_1T\otimes\cdots \otimes H_1T$ gives an
identification of $H_pT$ with the antisymmetric invariants
$\Lambda_pL\subseteq L^{\otimes p}$.   In terms of this identification the
Pontryagin product $H_1T\otimes 
H_1T\ra H_2T$ is given by $t\otimes t'\mapsto t\wedge t':=t\otimes
t'-t'\otimes t\in \Lambda_2 L$ (a direct consequence of the fact that
$H_*T$ is a \emph{graded} Hopf algebra: $\psi(tt')=\psi(t)\psi(t')= (t\otimes
1+1\otimes t)(t'\otimes 1+1\otimes t')=tt'\otimes 1 + (t\otimes
t'-t'\otimes t)+1\otimes tt'$).

The Kronecker pairing $(-,-)\colon H^*(T;B)\otimes H_*T\ra
B$ then gives an identification $H^p(T;B) \xra{\sim}
\Hom(H_pT,B)=\Hom(\Lambda_p L,B )$.   We note the following formula
for the cup product in these terms, which involves a tricky sign.
\begin{prop}\label{prop:cup-product-sign}
Let $f\in H^1(T;B)$ and $f'\in H^1(T;B')$ be cohomology classes
corresponding to $m\in \Hom(L,B)$ and $m'\in 
\Hom(L,B')$ via the Kronecker pairing.  Then with respect to the
Kronecker pairing, the cup product $f\smile f'\in H^2(T;B\otimes B')$
corresponds to  
\[
(\Lambda_2L \hookrightarrow L\otimes L \xra{-m\otimes m'}
B\otimes B') \in \Hom(\Lambda_2L, B\otimes B').
\]
Thus, $f\smile f'$ corresponds to the function $t\wedge t' \mapsto
(-m\otimes m')(t\wedge t')=
-m(t)\otimes m'(t')+m(t')\otimes m'(t)$.  
\end{prop}
\begin{proof}
Using the \emph{graded} Kronecker pairing $(-,-)\colon
(H^*(T;B)\otimes 
H^*(T;B'))\otimes (H_*T\otimes H_*T)\ra B\otimes B'$ we have
\[
(f\smile f',u) = (f\otimes f',\sum v\otimes v') =
-\sum (f,v)(f',v'),
\]
for $u\in H_2T$ where $\psi(u)=\sum v\otimes v' \in H_1T\otimes
H_1T$ (the component of the coproduct in degree $(1,1)$).   In terms
of our identifications, $\psi\colon H_2T\ra H_1T\otimes H_1T$ is the
inclusion $\Lambda_2L\ra L\otimes L$, and the formula follows.

(The additional sign here comes from a conflict of two sign
conventions: the graded Kronecker pairing $H^1(T;B)\otimes
H^1(T;B')\otimes H_1T\otimes H_1T\ra B\otimes B'$, which is what is
used to identify the coproduct $\psi$ as dual to cup product, and
which introduces a sign,  vs.\ the
evaluation pairing $\Hom(\Lambda_2L,B)\otimes
\Hom(\Lambda_2L,B')\otimes L\otimes L\ra B\otimes B'$, which does not
introduce a sign.)
\end{proof}

The Kronecker pairing generalizes to the ``slant product''
\[
f,v\mapsto f\nabla v \colon H^{p+q}(T;M)\times H_qT \ra H^p(T;M)
\]
by
\[
f\nabla v = \sum f'(f'',v),
\]
where $f\in H^{p+q}(T;M)$, $v\in H^qT$, and $\sum f'\otimes
f''= \mult_* f$, the image of $f$ under the map $H^*(T;M)\ra
H^*(T\times T; M)\approx H^*(T;M)\otimes H^*T$ induced by
multiplication in $T$.  Thus  if $\len{f}=\len{v}$ then $f\nabla
v=(f,v)$.

Given $n\in \Hom(\Lambda_kL, M)$ and
$t\in L$, we define the contraction operation $n\nabla t \in
\Hom(\Lambda_{k-1}L,M)$ by  
\[
(n\nabla t)(\tau) := n(t\wedge \tau).
\]
\begin{prop}\label{prop:contraction-pairing}
With respect to the usual identifications $H_1T=L$ and
$H^k(T;M)=\Hom(\Lambda_k L, M)$, the slant product coincides with the
contraction pairing.
\end{prop}
\begin{proof}
Let $f\in H^k(T;M)$ so $n=(f,-)\colon H_kT=\Lambda_k L\ra M$.  For
$t\in H_1T=L$ and $u\in H_{k-1}T=\Lambda_{k-1}L$ we have
\begin{align*}
(f\nabla t,u) &= (\sum f'(f'',t),u)= \sum
(f',u)(f'',t)
=(-1)^{k-1}\sum(f'\otimes f'', u\otimes t)
\\
&= (-1)^{k-1}(f, \mult_*(u\otimes t))= (-1)^{k-1}(f, u\wedge
  t)=(f,t\wedge u) = n(t\wedge u) = (n\nabla t)(u),
\\
\end{align*}
so $f\nabla v$ corresponds to $n\nabla v$.
\end{proof}
For instance, if $k=2$, then $H^2(T;M)\otimes H_1T\ra H^1(T;M)$ is
described by     
$(n\nabla t)(t')=n(t\wedge
t') = n(t\otimes t'-t'\otimes t)$. 

Finally, given a bilinear map $\gamma\colon L\otimes L\ra C$, we
will use the same symbol $\gamma$ for its restriction $\Lambda_2 L\ra
C$ (e.g.,  $\gamma=\omega(m\otimes m')$ in the statement of the 
\eqref{thm:general-thm} below).

\subsection{The homotopy groups of $BW_0^T(G_\omega)$}

We now describe $\pi_*BW^T(G_\omega)$, its quadratic invariant,  the evident
action of $\ol{W}(G_\omega)$ on homotopy groups, and its cohomology
ring.  After this we briefly explain how \eqref{thm:main-thm} is read
off from this calculation.

\begin{thm}\label{thm:general-thm}
Let $\omega\colon B\otimes B\ra C$ be a bilinear  function with
associated quadratic function $\phi$ and Hessian form $\beta$, and let
$G_\omega$ be a topological group associated with $\omega$ as in
\eqref{subsec:explicit-group-model}. 
The space $X=BW^T_0(G_\omega)$ is an object of $\mc{C}$, with 
\begin{align*}
  \pi_3 X &\approx C,
\\
  \pi_2 X &\approx L\times B\times \Hom(L,C),
\end{align*}
and with quadratic invariant $\phi^\sharp\colon \pi_2X\ra \pi_3X$
given by 
\[
\phi^\sharp(t,y,x) = \phi(y) + xt,\qquad t\in L,\; y\in B,\; x\in \Hom(L,C).
\]
Furthermore we have
\[
  \ol{W}^T(G_{\omega}) \approx  \Aut(L)\ltimes E,
\]
where $E$ is a group with underlying set 
\[
E= \Hom(L,B)\times \Hom(\Lambda_2 L, C)
\]
and group law
\[
(m,n)\cdot (m',n') = (m+m', \omega(m\otimes m') + n + n'), \qquad
m,m'\in \Hom(L,B),\; n,n'\in \Hom(\Lambda_2L, C).
\]
while the semi-direct product is defined via the action of $\Aut(L)$
on $E$ given by
\[
A\propto (m,n) = (mA^{-1}, n(\Lambda_2A^{-1})). 
\]
The action of $E\subseteq \ol{W}^T(G_{\omega})$ on $\pi_*X$ is given by 
\begin{align*}
  (m,n) \propto c &= c, & c\in \pi_3X
\\
  (m,n)\propto (t,y,x) &= (t, y+mt, x- \beta(y,m) - \omega(mt, m) +
  n\nabla t), & (t,y,x)\in \pi_2X,
\end{align*}
while the action of $\Aut(L)\subseteq \ol{W}^T(G_{\omega})$ on
$\pi_*X$ is given by  
\begin{align*}
  A \propto c &= c, & c\in \pi_3X,
\\
  A\propto (t,y,x) &= (At,y,xA^{-1}), & (t,y,x)\in \pi_2X.
\end{align*}
\end{thm}

\begin{rem}
The central extension $E$ corresponds to the cocycle $\gamma\colon
\Hom(L,B)\times \Hom(L,B)\ra \Hom(\Lambda^2L,C)$ defined by
$(m,m')\mapsto \omega(m\otimes m')$.  Up  to isomorphism, this central
extension depends only on the antisymmetrization of $\gamma$, which
satisfies $\gamma_\antisym(m,m')=\beta(m\otimes m')-\beta(m'\otimes
m)$ and so depends only on the quadratic function $\phi$. 
\end{rem}

The proof of \eqref{thm:general-thm} is given in the next several
sections, culminating in \eqref{subsec:proof-general-thm}.

It is straightforward to check that $\phi^\sharp$ is regular.  For
instance, if we choose coordinates and write $t=(t_1,\dots,t_r)\in
L=\Z^r$, $(y_1,\dots,y_d)\in 
B=\Z^d$, and $(x_{11},\dots,x_{re})\in \Hom(L,C)=\Z^{r\times e}$, and
write $\phi=(\phi_1,\dots,\phi_e)$, then 
$\phi^\sharp=(\phi^\sharp_1,\dots, \phi^\sharp_e)$  where 
$\phi_k^\sharp(t,y,x) = \phi_k(y)+ t_1x_{1k}+\cdots + t_rx_{rk}$.
This sequence of polynomials is easily seen to be regular (when $r\geq
1)$; in fact, it 
remains 
regular  after passing to the quotient ring in which
$y_1=\cdots =y_d=0$. 
From \eqref{prop:cohomology-23-space} we obtain  the following.
\begin{cor}
If the rank of 
$T$ is positive, then 
\begin{align*}
H^*(BW^T_0(G_{\omega});\Q) &\approx \Sym((L\times B\times
\Hom(L,C))^*\otimes \Q)/(\text{image of $(\wt\phi^\sharp)^*$})
\\
&\approx
\Q[t_1,\dots,t_r,y_1,\dots,y_d,x_{11},\dots,x_{re}]/(\phi_1^\sharp,\dots,
\phi_r^\sharp),
\end{align*}
with the evident action by $\ol{W}^T(G_{\omega})^\op$.  
\end{cor}

We obtain the statement of the main theorem \eqref{thm:main-thm} by: 
\begin{itemize}
\item  setting $B=\Z^d$, so $K(B,1)\approx U(1)^d$,
\item setting $C=\Z$,
\item taking $\omega\colon B\otimes B\ra C$ to be any quadratic
  refinement of $\phi\colon B\ra C$,
\item setting $T=\T^2$ so $L=H_1\T^2=\Z^2$, and
\item  identifying $\Z\approx \Lambda_2L$ via the generator
  $e_1\wedge e_2\in \Lambda_2L$, where $e_1=(1,0), e_2=(0,1)\in
  \Z^2=L$.
\end{itemize}
This last identification implies that for $n\in \Z\approx
\Hom(\Lambda^2L, \Z)$ and $t=(t_1,t_2)\in L$, we have that $(n\nabla
t)(e_1)= -t_2$ and $(n\nabla t)(e_2)=t_1$.  With these choices,
$W^T(G_\omega)$ is a model for $\mc{W}^\Sigma(U(1)^d
\times_\phi K(\Z,2))$, and the results of \eqref{thm:main-thm} follow.

\section{Computation of $\pi_*\Map(\Sigma, G_\omega)$}
\label{sec:compute-map-sigma-g}

We fix a topological group $G=G_\omega$ associated to a homomorphism
$\omega\colon B\otimes B\ra C$.   In this section we compute
invariants of the space $\Map(\Sigma,G_\omega)$ for an arbitrary space
$\Sigma$.

\subsection{Bilinear cohomology operations}

Any bilinear map $\alpha\colon B\otimes B'\ra C$
induces a bilinear cohomology operation
\[
\wt\alpha\colon H^p(-;B)\times H^q(-;B')\ra H^{p+q}(-;C)
\]
by 
\[
\wt\alpha(x,y) := \alpha_*(x\smile y),
\]
where the cup product is $\smile\colon H^p(-;B)\times H^q(-;B')\ra
H^{p+q}(-;B\otimes B')$, and 
 $\alpha_*\colon H^*(-;B\otimes B')\ra H^*(-;C)$ is the natural
map induced by the homomorphism on coefficients.  

\begin{rem}\label{rem:bilin-op-symmetry}
We note that
\[
\wt\alpha(x,y) = (-1)^{\len{x}\len{y}}\wt{(\alpha\tau)}(y,x),
\]
where $\tau\colon B'\otimes B\xra{\sim} B\otimes B'$ is the evident
transposition map.  In particular, if $\alpha\colon B\otimes B\ra C$
is symmetric, then
$\wt\alpha(x,y)=(-1)^{\len{x}\len{y}}\wt\alpha(y,x)$.  
\end{rem}

\subsection{The functors $\pi_*\Map(-,G_\omega)$}

We are going to describe
$\pi_k\Map(-,G_\omega)$ (basepoint at the identity
element) as functors on the homotopy category of
spaces, together with their Samelson products.

\begin{prop}\label{prop:map-into-g}
Let $\Sigma$ be an arbitrary space. 
\begin{enumerate}
\item
We have natural bijections of sets
\[
\pi_k \Map(\Sigma, G_\omega) \approx H^{1-k}(\Sigma;B)\times H^{2-k}(\Sigma;C)
\]
for all $k$.
\item The group law on $\pi_k\Map(\Sigma, G_\omega)$ for $k\geq1$ is
  the evident additive one.  
The group law on $\pi_0$ is given by
\[
(u,v)\cdot (u',v') = (u+u', -\wt\omega(u,u')+v+v').
\]
\item
Samelson products $\pi_p\times \pi_q\ra \pi_{p+q}$ on
$\Map(\Sigma,G_\omega)$ are given, in terms of the above bijection, by 
\[
\pairing{(y,x)}{(y',x')} = (0, -(-1)^{p(1-q)}\wt{\omega_\sym}(y,y')).
\]

\end{enumerate}
\end{prop}

Of course,  $\pi_k\Map(\Sigma,G_\omega)\approx0$ for $k\geq 3$, and
the first factor is always $0$ in 
$\pi_2\Map(\Sigma,G_\omega)$.  In
low dimensions, the Samelson products take the form
\begin{align*}
\pairing{(u,v)}{(y,x)} &= (0,
                         -\wt{\omega_\sym}(u,y))=(0,-\wt{\omega_\sym}(y,u)),
& \pi_0\times\pi_1\ra \pi_0,
\\
\pairing{(y,x)}{(u,v)} &= (0, \wt{\omega_\sym}(y,u))
=(0,\wt{\omega_\sym}(u,y)), & \pi_1\times \pi_0\ra \pi_1,
\\
\pairing{(y,x)}{(y',x')} &= -\wt{\omega_\sym}(y,y') =\wt{\omega_\sym}(y',y),
& \pi_1\times \pi_1\ra \pi_2.
\end{align*}
The Samelson product $\pi_0\times \pi_2\ra \pi_2$ is trivial. 

\begin{rem}\label{rem:inversion-map-to-g}
We record here the inversion formula for $\pi_0\Map(\Sigma,G_\gamma)$:
\[
(u,v)^{-1} = (-u, -\wt\omega(u,u)-v).
\]
\end{rem}

\begin{rem}
\label{rem:sign-in-pi0-group-law}
Suppose $\Sigma=T$ is a torus so  that as in
\eqref{subsec:torus-conventions} we have 
identifications $L=H_1T$, and 
\[
H^1(\Sigma;B)\xra{\sim} \Hom(L, B),\qquad H^2(\Sigma;C)\xra{\sim}
\Hom(\Lambda_2L, C).
\]
Then 
\[
\pi_0 \Map(T, G_\omega) \approx \Hom(L,B)\times \Hom(\Lambda_2 L, C)
\]
with group law given by 
\[
(m,n)\cdot (m',n') = (m+m', \omega(m\otimes m')+n+n'),\qquad m,m'\in
\Hom(L,B), \; n,n'\in \Hom(\Lambda_2L,C),
\]
where $\omega(m\otimes m')$ represents the composite $\Lambda_2L
\hookrightarrow L\otimes L \xra{m\otimes m'} B\otimes B \xra{\omega}
C$.   The change sign in the formula is a consequence of
\eqref{prop:cup-product-sign}. 
\end{rem}

\begin{rem}
Let $X=B\Map(T,G_\omega)$.  Then we have $\pi_k
X=\pi_{k-1}\Map(T,G_\omega)$.
The identifications of \eqref{prop:map-into-g} give  
\[
\pi_1 X \approx \Hom(L,B)\times \Hom(\Lambda_2 L,C),\qquad
\pi_2 X \approx B\times \Hom(L,C), \qquad
\pi_3X \approx C,
\]
with group law on $\pi_1X=\pi_0\Map(T,G_\omega)$ given as above.  

Tracing through: the relation between Samelson products and Whitehead
products \eqref{subsec:samelson-products}, the expression of
$\pi_1\curvearrowright \pi_n$ in 
terms of Whitehead products \eqref{subsec:whitehead-products}, and the
relation between 
Whitehead products and the quadratic invariant \eqref{subsec:sc-3-types}, 
together with the identifications and formulas of
\eqref{prop:map-into-g}, we 
find that: 
\begin{itemize}
\item The action of $\pi_1 X$ on
  $\pi_2X$ is given by
\[
(m,n)\propto (y,x) = (y,x) + [(m,n),(y,x)]= (y,x) +
\pairing{(m,n)}{(y,x)} = (b, x-\beta(y,m)).
\]
\item The action of $\pi_1 X$ on  $\pi_3X$ is
  trivial.
\item The quadratic invariant $\phi^\sharp\colon \pi_2X\ra \pi_3X$ is
  given by
$\phi^\sharp(y,x)=\phi(y)$.  
\end{itemize}
\end{rem}

\begin{proof}[Proof of \eqref{prop:map-into-g}(1) and (2)]
The bijections of (1) are  immediate from the description of
$G_\omega$.  The formula for the group law
in (2) amounts to the fact that the topological cocycle $-\kappa \colon
K(B,1)\times K(B,1)\ra K(C,2)$
represents the cohomology operation $-\wt\omega$.  That the group
structure for $\pi_{*\geq 1}$ is the additive one is straightforward;
i.e., $\Omega G\approx K(B,0)\times K(C,1)$ as an $H$-space.  
\end{proof}

It remains to prove part (3) of \eqref{prop:map-into-g}, for which we
need a suspension map.

\subsection{The suspension map}

There is a natural \dfn{suspension map}
\[
S_k\colon \pi_p\Map(\Sigma, X) \ra \pi_{p-k}\Map(\Sigma\times
S^k, X) 
\]
induced by  the evident inclusion
\[
\Map_*(S^k,\Map(\Sigma,X))\subseteq
\Map(S^k,\Map(\Sigma,X))\approx \Map(\Sigma\times S^k, X). 
\]
We compute the effect of suspension for $X=G_\omega$ in terms of  the
isomorphisms 
\eqref{prop:map-into-g}(1).  
\begin{prop}\label{prop:suspension-map-formula}
When $X=G_\omega$, the suspension map $S_k$ is given by 
\[
S_k(x,y) = (x\times \epsilon_k, y\times \epsilon_k), \qquad x\in
H^{1-p}(\Sigma;B),\; y\in H^{2-p}(\Sigma;C),
\]
where $\epsilon_k\in H^kS^k$ is the canonical generator.  In
particular, the suspension map is injective.
\end{prop}
\begin{proof}
Using $G_\omega=K(B,1)\times K(C,2)$, we see that $S_k$ is the same as the
suspension map in cohomology.  (See \eqref{subsec:cup-product}.)
\end{proof}

\subsection{Computation of  Samelson products, and proof of
  \eqref{prop:map-into-g}(3)}

When $p=q=0$, the Samelson product is just the commutator on
$\pi_0\Map(\Sigma, G_\omega)$.  Thus, in terms of \eqref{prop:map-into-g}(1), 
\[
\pairing{(y,x)}{(y',x')} = (0,-\wt\omega(y,y')+\wt\omega(y',y)) =
(0,-\wt{\omega_\sym}(y,y')). 
\]
The last equality is because 
$\wt{\gamma_\sym}\colon H^1(\Sigma;B)\times H^1(\Sigma;B)\ra
H^2(\Sigma;C)$ is the \emph{antisymmetrization} of $\wt\gamma$, for
degree reasons \eqref{rem:bilin-op-symmetry}. 

\begin{prop}\label{prop:commutator-and-suspension}
Let $G$ be a topological group.  
There is a commutative diagram of the form
\[\xymatrix{
{\pi_p\Map(\Sigma,G)\times \pi_q\Map(\Sigma,G)}
\ar[r]^-{\pairing{-}{-}} \ar[d]_{S_p\times S_q}
& {\pi_{p+q}\Map(\Sigma,G)} \ar[d]^{S_{p+q}} 
\\
{\pi_0\Map(\Sigma\times S^p,G)\times \pi_0\Map(\Sigma\times S^q,G)}
\ar[d]_{\pi_p^*\times\pi_q^*} 
& {\pi_0\Map(\Sigma\times S^{p+q}, G)} \ar[d]^{\pi_{p,q}^*}
\\
{\pi_0\Map(\Sigma\times S^p\times S^q,G)\times \pi_0\Map(\Sigma\times 
  S^p\times S^q,G)}\ar[r]_-{\pairing{-}{-}}
& {\pi_0\Map(\Sigma\times S^p\times S^q,G)}
}\]
where $\pi_p\colon S^p\times S^q\ra S^p$, $\pi_q\colon S^p\times
S^q\ra S^q$, and $\pi_{p,q}\colon S^p\times S^q\ra S^p\sm S^q\approx
S^{p+q}$ are the evident projections.
\end{prop}
\begin{proof}
The underlying point-set diagram commutes.
\end{proof}

\begin{proof}[Proof of \eqref{prop:map-into-g}(3)]
We have already proved the commutator formula for $\pi_0$ above.
Inserting  $(y,x)\in \pi_p\Map(\Sigma,G)$ and  $(y',x')\in
\pi_q\Map(\Sigma,G)$, going around the upper/right side of the square
of \eqref{prop:commutator-and-suspension} gives
\[
 \pairing{(y,x)}{(y',x')} \times \epsilon_p\times \epsilon_q,
\]
while going around the left/lower side gives
\begin{align*}
  (0,-\wt{\omega_\sym}(y\times \epsilon_p\times 1, y'\times1 \times\epsilon_q))
&= (0,-(-1)^{p(1-q)}\wt{\omega_\sym}(y,y') \times \epsilon_p\times \epsilon_q),
\end{align*}
since $y'\in H^{1-q}(\Sigma;B)$, using the formula for $\pi_0$.
Therefore we arrive at the formula  
\[
\pairing{(y,x)}{(y',x')} = (0,-(-1)^{p(1-q)}\wt{\omega_\sym}(y,y')),\qquad
(y,x)\in \pi_p\Map(\Sigma,G), \; (y',x')\in \pi_q\Map(\Sigma, G).
\]
\end{proof}

\subsection{A desuspension map}

Consider the basepoint preserving map
\[
D\colon \Map( \Sigma\times S^1, G)\ra \Omega\Map(\Sigma, G)
\]
defined by 
\[
(Df)(t)(s):= f(s,t)\cdot f(s,*)^{-1}, \qquad t\in S^1,\; s\in \Sigma,
\]
using the group law of $G$; here $*\in S^1$ represents the
basepoint.  This induces a map on homotopy 
groups
\[
D_*\colon \pi_k\Map(\Sigma\times S^1, G) \xra{\pi_k D} \pi_k\Omega
\Map(\Sigma, G) \xra[\sim]{\nu} \pi_{k+1}\Map(\Sigma, G).
\]
\begin{prop}\label{prop:desuspension-formula}
For $k=0,1$ the  map $D_*\colon \pi_k\Map(\Sigma\times S^1, G_\omega)\ra
\pi_{k+1}\Map(\Sigma, G_\omega)$ is given by
\[
D_*(y\times 1 + y'\times\epsilon,\; x\times 1+ x'\times\epsilon) =
(y',\; -\wt\omega(y',y)+x'),
\]
where $y\times 1+y'\times \epsilon\in H^{1-k}(\Sigma\times S^1;B)$,
$x\times 1+x'\times \epsilon\in 
H^{2-k}(\Sigma\times S^1;C)$, and where $\epsilon \in
H^1S^1$ is such that $(\epsilon, [S^1])=1$. 

In this case $k=1$ we must have  $y'=0$ and this simplifies to 
\[
D_*(y\times 1,\; x\times 1+x'\times \epsilon) = x'.
\]
\end{prop}

\begin{proof}
Consider the composite  $S_1\circ D$ with the suspension map, which is
the self-map of 
$\Map(\Sigma\times S^1,G)$ which sends $f$ to  $(s,t)\mapsto
f(s,t)\cdot f(s,*)^{-1}$.  

For $k=0$, the effect of
$S_1\circ D$ on $(y+y'\times\epsilon,
x+x'\times\epsilon)$ (we omit ``$\times 1$'' from the notation), using
the formula for the group law on 
$\pi_0$, including the formula \eqref{rem:inversion-map-to-g} for
inversion,  is
\begin{align*}
(y+y'\times \epsilon, x+x'\times\epsilon)\cdot
  (y,x)^{-1}
&= (y+y'\times \epsilon, x+x'\times\epsilon)\cdot
  (-y,-\wt\omega(y,y)-x)
\\
&= (y'\times\epsilon, \wt\gamma(y,y) -\wt\omega(y+y'\times\epsilon,
  -y) +x'\times\epsilon)
\\
&= (y'\times\epsilon, \bigl(-\wt\omega(y',y) +x')\times
  \epsilon). 
\end{align*}
Note that $\wt\omega(y'\times\epsilon,
y)=\omega_*((y'\times\epsilon)\smile y) = -\omega_*(y'\smile y)\times
\epsilon = -\wt\omega(y', y)\times \epsilon$, since $y\in H^1$.  
We read off the formula we need using
\eqref{prop:suspension-map-formula}. 

For $k=1$, the calculation of $S_1\circ D$ on $(y,
x+x'\times\epsilon)$ is as expected:
\[
  (y, x+x'\times \epsilon)\cdot(y,x)^{-1} = (y,x+x'\times\epsilon)-(y,x)=
(0,x'\times\epsilon),
\]
whence the desired formula.
\end{proof}

\section{Computation of $\pi_*W^T(G_\omega)$ and proof of
  \eqref{thm:general-thm}} 
\label{sec:proof-of-theorem}

\subsection{Certain Samelson products in $\pi_*W^T(G_\omega)$}

For a based map $v\colon S^1\ra T$, define
\[
C_v\colon T\times S^1\ra T,\qquad C_v(t,s) = v(s)^{-1}\cdot t = t\cdot v(s)^{-1}.
\]
That is, $C_v$ is the composite 
\[
T\times S^1 \xra{\id\times v} T\times T\xra{\id\times\inv}
T\times T \xra{\mult} T.  
\]
\begin{lemma}
\label{lemma:cv-formula}
The induced map $C^*_v\colon H^k(T;M)\ra H^k(T\times S^1;M)$ is given
by 
\[
C^*_v(f) = f\times 1 - (f\nabla v)\times \epsilon,
\]
where we also write $v$ for $v_*[S^1]\in H_1T$.  
\end{lemma}
\begin{proof}
If $\mult^*f = \sum f'\otimes f''$, then
\[
(\id \times \inv)^*\mult^* f = \sum(-1)^{\len{f''}} f'\otimes f''.
\]
Since $v^*f=(f,v_*[*])\times1 +(f,v_*[S^1])\times \epsilon$, we have
\[
(\id\times v\inv)^*\mult^* f = \sum(-1)^{\len{f''}} f'(f'',[*])\times
1 + (-1)^{\len{f''}}f'(f'',v)\times \epsilon = f\nabla[*]\times 1 -
f\nabla v\times \epsilon.
\]
\end{proof}

Given elements $(1,f), (g,1) \in \Map(T,G)\rtimes T \subset W^T(G)$,
their commutator is given by 
\[
(1,f)\cdot (g,1)\cdot (1,f^{-1})\cdot (g^{-1},1) 
= (1, t\mapsto g(f^{-1}\cdot t)\cdot g^{-1}(t)).  
\]
\begin{lemma}\label{lemma:samelson-t-with-map-g}
For $v\in \pi_1T$ and $w\in \pi_k\Map(T,G)$ we have  the Samelson product
\[
\pairing{(0,v)}{(w,0)} = \bigl((D\circ C^*_v)_*(w),0\bigr) \in
\pi_{k+1} \Map(T,G)\rtimes T = \pi_{k+1}\Map(T,G)\times \pi_{k+1}T,
\]
where $D\colon \Map(T\times S^1,G)\ra \Omega\Map(T,G)$ is the desuspension
map, and $C_v^*=\Map(C_v,\id)\colon \Map(T,G)\ra \Map(T\times S^1,G)$.  
\end{lemma}
\begin{proof}
Evidently $D\circ C_v^*$ sends $g\in \Map(T,G)$ to $h\in
\Omega\Map(T,G)$ defined by $h(s):=\bigl(t\mapsto g(v(s)^{-1}\cdot
t)\cdot g(t)^{-1}\bigr)$.  That is, $h(s)=v(s)\cdot g\cdot
v(s)^{-1}\cdot g^{-1}$, the commutator in $\Map(T,G)\rtimes T$ of
$v(s)\in T$ and $g\in \Map(T,G)$.  Thus for a based map $w\colon
S^k\ra \Map(T,G)$, the composite $D\circ C_v^*\circ w$ is adjoint to
the desired Samelson product.
\end{proof}

Now fix $G=G_\omega$.  We compute the Samelson product
$\pairing{-}{-}\colon \pi_1 W^T(G_\omega)\times \pi_k W^T(G_\omega)\ra
\pi_{k+1}W^T(G_\omega)$ on elements which come from the subgroups $T$
and $\Map(T,G)$.  
\begin{prop}\label{prop:samelson-t-with-map-g}
For $t\in L=\pi_1T\subset \pi_1D(T)$ and $(y,x)\in H^{1-k}(T;B)\times
H^{2-k}(T;C) = \pi_k\Map(T,G_\omega) \subset \pi_1W^T(G_\omega)$, then
for $k=0,1$ we
have 
\[
\pairing{(0,t)}{((y,x),0)} = (-y\nabla t,\; \wt\omega(y\nabla t, y) -
x\nabla t) \in \pi_{k+1}\Map(T,G_\omega) \subset \pi_{k+1} W^T(G_\omega).  
\]
For $k=1$ this simplifies to 
\[
\pairing{(0,t)}{((y,x),0)}  = (0, -x\nabla t).
\]
\end{prop}
\begin{proof}
Combine \eqref{lemma:samelson-t-with-map-g}, the formula 
\eqref{prop:desuspension-formula} for the desuspension map, and the effect of $C_t^*$ on
$\pi_*\Map(T,G_\omega)$ (for  $t\in L=\pi_1T$),
which is computed by 
\eqref{lemma:cv-formula}.   Explicitly, in terms of the isomorphism
$\pi_k\Map(\Sigma, G_\omega)\approx H^{1-k}(\Sigma;B)\times
H^{2-k}(\Sigma; C)$ of \eqref{prop:map-into-g}, the formula of
\eqref{lemma:cv-formula} gives
\[
C_t^*(y,x) = (y\times 1 - (y\nabla t)\times \epsilon, x\times
1-(x\nabla t)\times \epsilon),
\]
whence   \eqref{prop:desuspension-formula} gives
\[
D(C_t^*(y,x)) = (-y\nabla t,\;-\wt\omega(-y\nabla t, y) -x\nabla
t)=(-y\nabla t,\; \wt\omega(y\nabla t,y)-x\nabla t).
\]
\end{proof}

\subsection{Structure of $\pi_* W^T(G_\omega)$}
\label{subsec:proof-general-thm}

We now put this together to compute the homotopy groups of
$\Map(T,G_\omega)\rtimes T\subset W^T(G_\omega)$ together with 
its Samelson products.
\begin{thm}\label{thm:homotopy-K-with-action}
Fix $\omega\colon B\otimes B\ra C$, and set $L=H_1T=\pi_1T$.  
\begin{enumerate}
\item
We have set bijections
\begin{align*}
  \pi_2(\Map(T,G_\omega)\rtimes T) &\approx \qquad\qquad\qquad\quad\quad\; H^0(T;C),
\\
  \pi_1(\Map(T,G_\omega)\rtimes T) &\approx 
H_1T \times H^0(T;B)\times  H^1(T;C),
\\
  \pi_0(\Map(T,G_\omega)\rtimes T)&  \approx \quad\quad\;\;\;\;\, H^1(T;B)\times 
     H^2(T;C).
\end{align*}
\item
The group structure on $\pi_1$ and $\pi_2$ is
the evident additive one, while the group law in $\pi_0$ is given by 
\[
(m,n)\cdot (m',n') = (m+m', -\wt\omega(m,m')+n+n').
\]
\item 
Samelson products on $\pi_*$ are  given in terms of the above
bijections by 
\begin{align}
\pairing{ (m,n)}{(t,y,x)} &= (0,\; m\nabla t,\;
-\wt{\beta}(m,y)-\wt\omega(m\nabla t, m) + n\nabla t), 
&\pi_0\times \pi_1\ra \pi_1,
\\
\pairing{(t,y,x)}{(m,n)} &= (0,\; -m\nabla t,\;
\wt{\beta}(m,y)+\wt\omega(m\nabla t, m) - n\nabla t),
& \pi_1\times \pi_0\ra \pi_1,
\\
\pairing{(t,y,x)}{(t',y',x')} &= -\wt{\beta}(y,y')-x\Delta t'
  -x'\Delta t,
& \pi_1\times \pi_1\ra \pi_2,
\end{align}
while $\pi_0\times \pi_2\ra \pi_2$ is trivial.
\end{enumerate}
\end{thm}
\begin{proof}
This is a combination of what we have proved up to now.  The set
bijections are from the evident  product decomposition of
$\Map(T,G_\omega)\rtimes T$ and \eqref{prop:map-into-g}(1).  The group
structure in $\pi_0$ is by \eqref{prop:map-into-g}(2), while in higher
degrees it is follows easily from \eqref{prop:map-into-g}(2) since the
groups must be abelian.  

Part (3) is a consequence of the bilinearity of the Samelson product,
combined with \eqref{prop:map-into-g}(3), 
\eqref{prop:samelson-t-with-map-g}, and the fact that Samelson
products in $\pi_*T$ vanish since $T$ is abelian. 
\end{proof}

We can add in the action of $\Aut(L)$.

\begin{prop}\label{prop:homotopy-K-aut-T-action}
With respect to the identification of the homotopy groups of
$\Map(T,G_\omega)\rtimes T$ described in
\eqref{thm:homotopy-K-with-action}, the (left) action of $\Aut(L)$ on
$\Map(T,G_\omega)$ induces the following action on homotopy groups,
written in terms of the evident left actions of $\Aut(L)=\Aut(T)$ on $H^*(T;M)$.
\begin{align*}
  A\propto (m,n) &= ((A^{-1})^*m, (A^{-1})^*n),
& (m,n)\in H^1(T;B)\times H^2(T;C),
\\ 
  A\propto (t,y,x)&= (At, y, (A^{-1})^*x), 
& (t,y,x)\in L\times H^0(T;B)\times H^1(T;C),
\\
A \propto c &=  c,
& c\in H^0(T;C).
\end{align*}
\end{prop}
\begin{proof}
Immediate by functoriality.
\end{proof}

If we introduce the isomorphisms
\[
H^0(T;M)\approx M,\qquad H^1(T;M)\approx \Hom(L,M),\qquad
H^2(T;M)\approx \Hom(\Lambda_2L, M),
\]
then we have
\begin{align*}
  \pi_2(\Map(T,G_\omega)\rtimes T) &\approx C,
\\
  \pi_1(\Map(T,G_\omega)\rtimes T) &\approx L\times B\times
  \Hom( L, C),
\\
  \pi_0(\Map(T,G_\omega)\rtimes T) & \approx \Hom(L,B)\times
  \Hom(\Lambda_2L, C),
\end{align*}
and the formulas take the form
\begin{align*}
(m,n)\cdot (m',n') &= (m+m',\; \omega(m\otimes m') + n+n'),
\\
\pairing{(m,n)}{(t,y,x)} &= (0,\; mt,\; -\beta(m,y)-\omega(mt,m)+n\nabla t),
\\
\pairing{(t,y,x)}{(m,n)} &= (0,\; -mt,\; \beta(m,y)+\omega(mt, m) -
n\nabla t),
\\
\pairing{(t,y,x)}{(t',y',x')} &= -\beta(y,y')-x\nabla t' -x'\nabla t,
\\
A\propto (m,n)&= (mA^{-1}, n(\Lambda_2 A^{-1}))
\\
A\propto (t,y,x) &= (At,y,xA^{-1}).
\end{align*}
Here:
\begin{itemize}
\item the pairings $m,t\mapsto m\nabla t$ and $x,t\mapsto x\nabla t$
  are examples of Kronecker pairings $ H^1(T;M)\times H_1T\ra
  H^0(T;M)$, and so become  evaluation maps
  $\Hom(L,M)\otimes L\ra M$;
\item the pairing $n,t\mapsto n\nabla t\colon H^2(T;C)\times H_1T\ra
  H^1(T;C)$ becomes the contraction pairing $\nabla\colon
  \Hom(\Lambda_2 L,C)\times L\ra \Hom(L,C)$ 
  \eqref{prop:contraction-pairing}; 
\item 
the expression $-\wt\omega(m,m')$ becomes $\omega(m\otimes m')$ as
explained in \eqref{rem:sign-in-pi0-group-law};
\item
the action of $\Aut(L)=\Aut(T)$ on $H^k(T;M)$ becomes the evident
action on $\Hom(\Lambda_kL, M)$ defined by functoriality of
$\Lambda_k$ and composition.
\end{itemize}

\begin{proof}[Proof of \eqref{prop:homotopy-K-aut-T-action}]
This is read off from the calculations
\eqref{thm:homotopy-K-with-action},
\eqref{prop:homotopy-K-aut-T-action}, together with the relation
between Whitehead and Samelson products \eqref{eq:sam-whit-relation},
and the fact that Whitehead products $\pi_1\times \pi_k\ra \pi_k$
describe the action  $\pi_1\curvearrowright \pi_k$
\eqref{subsec:whitehead-products}. 
\end{proof}

We can now give the proof of the general result.
\begin{proof}[Proof of \eqref{thm:general-thm}]
We read off the results using the isomorphism
$\pi_*BW^T(\Sigma)\approx \pi_{*-1}W^T(\Sigma)$.  The Samelson
products become Whitehead products, up to a sign as described in
\eqref{eq:sam-whit-relation}.  
\end{proof}

\section{Conventions}
\label{sec:conventions}

\subsection{Orientations}\label{subsec:orientations}

Let $I=[0,1]$.  We use $S^n= I^n/\partial I^n$.  We orient $I^n$, and
thus  $S^n$, using the K\"unneth map, e.g., using  $I^n/\partial I^n 
\approx (I/\partial I)^{\sm n}$.   The boundary $\partial I^n$ is
homeomorphic to $S^{n-1}$, and its orientation is fixed by the boundary
map in singular homology.   When necessary, we take $(0,0)$ as the
base point of $\partial I^n$.

Examining the Eilenberg-Zilber map
shows that the face $\set{(t_1,\dots,t_n)}{t_i=1}\subset I^n$ receives
a positive 
orientation when $i$ is odd, and a negative orientation when $i$ is
even.  In particular, we see that an orientation preserving equivalence
$I/\partial I\ra \partial I^2$ goes counterclockwise around the
square; e.g., $(0,0)$ to $(1,0)$ to $(1,1)$ to $(0,1)$ to $(0,0)$.  

\subsection{Loops}

We compose paths in temporal order: thus $\gamma\cdot \delta$ is
defined when $\gamma(1)=\delta(0)$.  

The standard action of the fundamental group $\pi_1 X\curvearrowright
\pi_p X$ is from the 
\emph{left}.   We write it as $\gamma\propto\alpha$, $\gamma\in \pi_1X$ and
$\alpha\in \pi_p X$.  It is necessarily defined, in terms of the
``balloon on a string'' picture, by putting the \emph{end} of the string
$\gamma(1)$ at the basepoint of the $p$-sphere: i.e., via $S^p\ra
I\cup_{\{1\}} S^p\xra{(\gamma,\alpha)}X$.   

With this convention, we
may extend this action to a functor $\pi_p \colon \Pi_1X\ra \Ab$ from
the fundamental groupoid, for $p\geq2$.

We have for $k\geq 1$ a standard isomorphism
\[
\nu\colon \pi_k X\xra{\sim} \pi_{k-1}\Omega X,
\]
defined so that $f\colon S^k\ra X$ corresponds to $\nu f\colon
S^{k-1}\ra X$ given by $(\nu f)(x_1,\dots,x_{k-1})(t)=f(x_1,\dots,
x_{k-1},t)$.  This convention matches the most 
lexically convenient convention for the tensor-hom adjunction, i.e.,
$\Map_*(X\sm Y,Z)\approx \Map_*(X,\Map_*(Y,Z))$.  In particular, the
standard isomorphism
$\Map_*(X\sm S^p,Z) \approx \Map_*(X,\Omega^pZ)$ is 
determined in this way.

\begin{rem}\label{rem:convention-susp-loop}
If instead we consider $\nu'\colon \pi_kX\xra{\sim} \pi_{k-1}\Omega
X$, so that $\nu(f)(x_1,\dots,x_{k-1})(t) = f(t,x_1,\dots,x_{k-1})$,
then we have $(\nu' f) \sim (-1)^{k-1}(\nu f)$, with the sign coming
from the evident transposition $S^{k-1}\sm S^1\ra S^1\sm S^{k-1}$.
(The map 
$\nu'$ is the standard identification used in
\cite{whitehead-elements}, for 
instance.)
\end{rem}

\subsection{Eilenberg-Mac Lane spaces}

For us, an Eilenberg-Mac Lane space consists of a based space $K$
\emph{equipped} 
with a choice of isomorphism $A\xra{\sim} \pi_nK$ to its only
non-trivial homotopy group.  Such an object is unique up to canonical
homotopy equivalence (and in fact, up to contractible choice), and so
can be unambiguously denoted $K(A,n)$.  

We thus obtain a canonical weak equivalence $\Omega K(A,n)\approx
K(A,n-1)$, using the standard isomorphism $\nu\colon \pi_n\xra{\sim}
\pi_{n-1}\Omega$ to fix the identification of the homotopy group.

The identification $\wt{H}^n(X,A)\approx \pi_0\Map_*(X,K(A,n))$ is
fixed in the usual way, via a choice of tautological class $\iota\in
\wt{H}^n(K(A,n);A)\approx \Hom(H_n(K(A,n);A),A) \approx
\Hom(\pi_nK(A,n), A)$ corresponding to the identity map of $A$.
(Ultimately, this depends on the choice of the fundamental class in
$H_nS^n$, which defines the Hurewicz map.)

Using the standard identification $\Omega^pK(A,n)\approx K(A,n-p)$, we
obtain a standard natural isomorphism
\[
\pi_p \Map(X,K(A,n)) \approx \pi_0\Map(X, \Omega^pK(A,n)) \approx
H^{n-p}(X;A). 
\]
Similarly, we obtain a natural suspension isomorphism
\[
\wt{H}^{n-p}(X;A) \approx [X,\Omega^{p}K(A,n)]_* \approx [X\sm S^p,
K(A,n)]_* \approx \wt{H}^n(X\sm S^p; A).
\]

\subsection{Cup product}
\label{subsec:cup-product}

The cup product $\smile \colon H^m(X;A)\otimes H^n(X;B)\ra
H^{m+n}(X;A\otimes B)$ is represented by a map $\smile \colon
K(A,m)\sm K(B,n)\ra  K(A\otimes B,m+n)$, characterized by the property
that $\pi_mK(A,m)\otimes \pi_nK(B,n)\ra \pi_{m+n}K(A\otimes B,m+n)$
induces the identity map of $A\otimes B$.

Let $\epsilon \in H^1S^1$ be the tautological class, corresponding to
the map $S^1\ra K(\Z,1)$ sending the tautological class $\iota_1\in
\pi_1S^1$ 
to $1\in \Z=\pi_1K(\Z,1)$.  The map
\[
-\times \epsilon\colon \wt{H}^{n-1}(X;A) \ra \wt{H}^n(X\sm S^1;A) 
\]
coincides with the standard isomorphisms
\[
[X,K(A,n-1)]_*\approx [X,\Omega K(A,n)]_* \approx [X\sm S^1,
K(A,n)]_*,
\]
and thus with the suspension isomorphism described above. 

More precisely: using our standard identifications $K(A,n-1)\approx
\Omega K(A,n)$ and $H^k(X,B)\approx [X,K(B,k)]$, 
the composite
\[
[X, K(A,n-1)] \approx [X,\Omega K(A,n)] \ra [X\times S^1, K(A,n)]
\]
coincides with $-\times \epsilon\colon H^{n-1}(X;A)\ra H^n(X\times
S^1;A)$, whereas the composite
\[
[X, K(A,n-1)]\approx [X,\Omega K(A,n)] \ra [S^1\times X, K(A,n)]
\]
coincides with $(-1)^{n-1}(\epsilon\times -)\colon H^{n-1}(X;A)\ra
H^n(S^1\times X;A)$ (and not with $\epsilon\times -$ as one might
naively think).

\subsection*{Cohomology operations}

A based map $\psi\colon K(A,m)\ra K(B,n)$ induces a cohomology
operation $\psi\colon H^m(X;A)\ra H^n(X;B)$.  Taking loops gives a map
$\Omega\psi\colon K(A,m-1)\ra K(B,n-1)$ and a corresponding operation
$\Omega\psi\colon H^{m-1}(X;A)\ra H^{n-1}(X;B)$.  The diagram
\[\xymatrix{
{H^{m-1}(X;A)} \ar[r]^{\Omega\psi} \ar[d]_{\times\epsilon}
& {H^{n-1}(X;B)} \ar[d]^{\times\epsilon}
\\
{H^m(X\sm S^1;A)} \ar[r]_{\psi}
& {H^n(X\sm S^1;B)}
}\]
commutes, i.e.,
\[
\Omega\psi(x)\times \epsilon = \psi(x\times \epsilon).
\]
Note that the analogous diagram involving $\epsilon\times$ only
commutes \emph{up to  sign} depending on $m-n$, i.e., $(-1)^{n-1}\epsilon
\times \Omega\psi(x) = \psi((-1)^{m-1}\epsilon\times x)$.

A based map $\psi\colon K(A,p)\sm K(B,q)\ra K(C,n)$ determines a
cohomology operation $\psi\colon H^p(X;A)\times H^q(X;B)\ra H^n(X;C)$
in two variables.  Using the evident maps $\Omega X\sm Y\ra
\Omega(X\sm Y)$ and $X\sm \Omega Y\ra \Omega(X\sm Y)$, we can loop
$\psi$ in either of its two inputs, obtaining
\[
\Omega_1 \psi\colon K(A,p-1)\sm K(B,q)\ra K(C,n-1),\quad \Omega_2\psi
\colon K(A,p)\sm K(B,q-1)\ra K(C,n-1).
\]
The relation between these are given by 
\[
\Omega_1\psi(x,y)\times \epsilon = (-1)^{\len{y}}\psi(x\times
\epsilon, y),\quad   
\Omega_2\psi(x,y)\times \epsilon = \psi(x,y\times \epsilon), 
\]

\subsection{Groups and loop spaces}

Let $G$ be a topological group, and let $EG\ra BG$ be the universal
principal bundle.  We may regard $EG$ as having a \emph{left} action
by $G$.  

For a based space $(X,x_0)$, the path fibration $PX\ra X$ may be
defined by 
\[
PX = \set{\gamma\in \Map([0,1],X)}{\gamma(0)=x_0}.
\]
With this definition (with the free end of the path at $t=1$), the
composition law on $\Omega X$ extends 
naturally to an
``action'' $*\colon \Omega X\times PX\ra PX$.  

Taking $X=BG$, any lift $EG\ra P(BG)$ of the two projections gives
rise to a weak equivalence $G\ra \Omega G$, which furthermore is an
$H$-space map.

\subsection{Whitehead products}
\label{subsec:whitehead-products}

The Whitehead product $[-,-]\colon \pi_p\times \pi_q \ra \pi_{p+q-1}$
is defined via precomposition with 
\[
\bigl(\partial I^{p+q} \ra \partial I^{p+q}/\sim \bigr) = \bigl(I^p\times \partial I^q\cup \partial I^p\times
I^q \ra 
I^p/\partial I^p\vee I^q/\partial I^q\bigr),
\]
using the orientation convention of \eqref{subsec:orientations}. 
\begin{itemize}
\item If $p=q=1$, the Whitehead product is the loop commutator
  $[\gamma,\delta]= \gamma\delta\gamma^{-1}\delta^{-1}$.

\item If $\gamma\in \pi_1$ and $\alpha\in \pi_{q\geq2}$, then the
  Whitehead product is $[\gamma, 
  \alpha] = (\gamma\propto\alpha)-\alpha$.  

\item On inputs in dimensions $\geq2$, $[-,-]$ is bilinear.

\item In general, we have the commutation relation
  $[\beta,\alpha]=(-1)^{pq}[\alpha,\beta]$ for $\alpha\in \pi_p$ and
  $\beta\in \pi_q$.

\item We only need to care about the cases when  $p,q\leq 2$:
\begin{align*}
[\delta,\gamma] &=    [\gamma,\delta]^{-1}, && \gamma,\delta\in \pi_1,
\\
[\alpha,\gamma] &=   [\gamma,\alpha], && \gamma\in \pi_1,\; \alpha\in
\pi_2,
\\
[\beta,\alpha] &= [\alpha,\beta], && \alpha,\beta\in \pi_1.
\end{align*}
These conventions agree with those of
\cite{whitehead-elements}*{\S7.4}\footnote{Although he defines the
  products using disks, not cubes.  It is also unclear to me where he puts
  his basepoint; when $p,q\geq2$, the choice of basepoint ``shouldn't 
  matter''.}.
\end{itemize}

\subsection{Samelson products}
\label{subsec:samelson-products}

For a topological group $G$, the Samelson product
$\pairing{-}{-}\colon \pi_p\times \pi_q\ra \pi_{p+q}$ is defined in
the evident way using 
the commutator map $G\sm G\ra G$ sending $(x,y)\mapsto
xyx^{-1}y^{-1}$.  

The definition admits an extension to loop spaces.  For such spaces,
the Samelson product agrees with the Whitehead product up to a sign.
In fact, for $\alpha\in \pi_pX$ and $\beta\in \pi_q X$ we have that
\begin{equation}\label{eq:sam-whit-relation}
\nu[\alpha,\beta] = (-1)^{p-1}\pairing{\nu(\alpha)}{\nu(\beta)},
\end{equation}
according to \cite{whitehead-elements}*{7.10}\footnote{In fact, his
  formula uses $\nu'$  instead of
  $\nu$.  But this makes no difference, 
  as can be shown using  \eqref{rem:convention-susp-loop}.}.  In
particular, we 
have
\begin{align*}
  \pairing{\nu(\gamma)}{\nu(\delta)} &= \nu[\gamma,\delta], && \gamma,\delta\in
                                                  \pi_1,
\\
  \pairing{\nu(\gamma)}{\nu(\alpha)} &= \nu[\gamma,\alpha], && \gamma\in
                                                  \pi_1,\,\alpha\in
                                                  \pi_2,
\\
  \pairing{\nu(\alpha)}{\nu(\gamma)} &= -\nu[\alpha,\gamma]=-\nu[\gamma,\alpha], &&
   \gamma\in  \pi_1,\, \alpha\in \pi_2,
\\
  \pairing{\nu(\alpha)}{\nu(\beta)} &= -\nu[\alpha,\beta], &&
                                                              \alpha,\beta\in \pi_2. 
\end{align*}
We note the commutation relation 
\[
\pairing{\alpha}{\beta} = -(-1)^{pq}\pairing{\beta}{\alpha},\qquad
\alpha\in \pi_pG,\; \beta\in \pi_qG.
\]


\end{document}